\documentclass[10pt]{amsart}
\usepackage{pgf,tikz,pgfplots,color}
\pgfplotsset{compat=1.15}
\usepackage{mathrsfs} %mathscr
\usetikzlibrary{arrows}
\usepackage{fancyhdr}
\usepackage{lastpage}

\usepackage[margin=0.9in]{geometry}
\usepackage{hyperref}
\hypersetup{
    colorlinks=true,
    linkcolor=blue,
    citecolor=brown,
    filecolor=magenta,
    urlcolor=blue,
}
\usepackage{amsmath}
\usepackage{amsfonts}
\usepackage{amssymb}
\usepackage{amsthm}
\usepackage{setspace}
\usepackage{inputenc}
\usepackage{comment}
\usepackage[shortlabels]{enumitem}

\newtheorem{thm}{Theorem}
\newtheorem{cor}[thm]{Corollary} 
\newtheorem{prop}[thm]{Proposition}
\newtheorem{lemma}[thm]{Lemma}
\theoremstyle{definition}
\newtheorem{defn}[thm]{Definition}

\newtheorem{ques}[thm]{Question}

\newtheorem{innercustomgeneric}{\customgenericname}
\providecommand{\customgenericname}{}
\newcommand{\newcustomtheorem}[2]{%
  \newenvironment{#1}[1]
  {%
   \renewcommand\customgenericname{#2}%
   \renewcommand\theinnercustomgeneric{##1}%
   \innercustomgeneric
  }
  {\endinnercustomgeneric}
}

\newcustomtheorem{customprop}{Proposition}
\newcustomtheorem{customlemma}{Lemma}

\theoremstyle{remark}
\newtheorem{rem}[thm]{Remark}

\newtheorem*{ack}{Acknowledgment}
\newcommand{\Z}{\mathbb{Z}}
\newcommand{\R}{\mathbb{R}} 

\newcommand{\1}{\mathbf{1}}

\newcommand{\As}{\mathscr{A}}

\newcommand{\Bs}{\mathscr{B}}

\newcommand{\Ec}{\mathcal{E}}

\newcommand{\Fs}{\mathscr{F}}

\newcommand{\M}{\mathcal{M}}

\newcommand{\Nc}{\mathcal{N}}
\newcommand{\Ns}{\mathscr{N}}
\newcommand{\Pc}{\mathcal{P}}
\newcommand{\Qc}{\mathcal{Q}}

\newcommand{\Ss}{\mathscr{S}}

\newcommand{\eps}{\varepsilon}
\newcommand{\dist}{\operatorname{dist}}
\newcommand{\supp}{\operatorname{supp}}
\newcommand{\loc}{\mathrm{loc}}

\makeatletter
\def\@tocline#1#2#3#4#5#6#7{\relax
  \ifnum #1>\c@tocdepth % then omit
  \else
    \par \addpenalty\@secpenalty\addvspace{#2}%
    \begingroup \hyphenpenalty\@M
    \@ifempty{#4}{%
      \@tempdima\csname r@tocindent\number#1\endcsname\relax
    }{%
      \@tempdima#4\relax
    }%
    \parindent\z@ \leftskip#3\relax \advance\leftskip\@tempdima\relax
    \rightskip\@pnumwidth plus4em \parfillskip-\@pnumwidth
    #5\leavevmode\hskip-\@tempdima
      \ifcase #1
       \or\or \hskip 1em \or \hskip 2em \else \hskip 3em \fi%
      #6\nobreak\relax
    \hfill\hbox to\@pnumwidth{\@tocpagenum{#7}}\par% <---- \dotfill -> \hfill
    \nobreak
    \endgroup
  \fi}
\makeatother

	\title[Morrey-type Spaces on Lipschitz Domains]{Some Intrinsic Characterizations of Besov-Triebel-Lizorkin-Morrey-type Spaces on Lipschitz Domains}          

\author[]{Liding Yao} 
\address{Liding Yao, Department of Mathematics,
	The Ohio State University, Columbus, OH 43210} 
\email{yao.1015@osu.edu}


\begin{document}

\begin{abstract}
    We give Littlewood-Paley type characterizations for Besov-Triebel-Lizorkin-type spaces $\mathscr B_{pq}^{s\tau},\mathscr F_{pq}^{s\tau}$ and Besov-Morrey spaces $\mathcal N_{uqp}^s$ on a special Lipschitz domain $\Omega\subset\mathbb R^n$: for a suitable sequence of Schwartz functions $(\phi_j)_{j=0}^\infty$,
    \begin{align*}
        \|f\|_{\mathscr B_{pq}^{s\tau}(\Omega)}&\textstyle\approx\sup_{P\text{ dyadic cubes}}|P|^{-\tau}\|(2^{js}\phi_j\ast f)_{j=\log_2\ell(P)}^\infty\|_{\ell^q(L^p(\Omega\cap P))};\\
        \|f\|_{\mathscr F_{pq}^{s\tau}(\Omega)}&\textstyle\approx\sup_{P\text{ dyadic cubes}}|P|^{-\tau}\|(2^{js}\phi_j\ast f)_{j=\log_2\ell(P)}^\infty\|_{L^p(\Omega\cap P;\ell^q)};\\
        \|f\|_{\mathcal N_{uqp}^{s}(\Omega)}&\textstyle\approx\big\|\big(\sup_{P\text{ dyadic cubes}}|P|^{\frac1u-\frac1p}\cdot 2^{js}\|\phi_j\ast f\|_{L^p(\Omega\cap P)}\big)_{j=0}^\infty\big\|_{\ell^q}.
    \end{align*}
    
    We also show that $\|f\|_{\mathscr B_{pq}^{s\tau}(\Omega)}$, $\|f\|_{\mathscr F_{pq}^{s\tau}(\Omega)}$ and $\|f\|_{\mathcal N_{uqp}^{s}(\Omega)}$ have equivalent (quasi-)norms via derivatives: for $\mathscr X^\bullet\in\{\mathscr B_{pq}^{\bullet,\tau},\mathscr F_{pq}^{\bullet,\tau},\mathcal N_{uqp}^\bullet\}$, we have $\|f\|_{\mathscr X^s(\Omega)}\approx\sum_{|\alpha|\le m}\|\partial^\alpha f\|_{\mathscr X^{s-m}(\Omega)}$. 
    
    In particular $\|f\|_{\mathscr F_{\infty q}^s(\Omega)}\approx\sum_{|\alpha|\le m}\|\partial^\alpha f\|_{\mathscr F_{\infty q}^{s-m}(\Omega)}\approx\sup_{P}|P|^{-n/q}\|(2^{js}\phi_j\ast f)_{j=\log_2\ell(P)}^\infty\|_{\ell^q(L^q(\Omega\cap P))}$.
\end{abstract}
\maketitle

\section{Introduction}

Let $\Omega\subset\R^n$ be a \textit{special Lipschitz domain}, that is, $\Omega$ is of the form $\{(x',x_n):x_n>\rho(x')\}$ where $\rho:\R^{n-1}\to\R$ is a Lipschitz function such that $\|\nabla\rho\|_{L^\infty}<\infty$. (See also \cite[Definition 1.103]{TriebelTheoryOfFunctionSpacesIII}.)

In \cite{ExtensionLipschitz}, based on the construction of his extension operator, Rychkov gave a Littlewood-Paley type intrinsic characterization of the Triebel-Lizorkin spaces on $\Omega$: for $0<p<\infty$, $0<q\le\infty$ and $s\in\R$, $\Fs_{pq}^s(\Omega)$ has the following equivalent (quasi-)norm (see \cite[Theorem 3.2]{ExtensionLipschitz}):
\begin{equation}\label{Eqn::Intro::IntroLPChar}
    f\mapsto\|(2^{js}\phi_j\ast f)_{j=0}^\infty\|_{\ell^q(\Z_{\ge0};L^p(\Omega))}=\bigg(\int_\Omega\Big(\sum_{j=0}^\infty 2^{jsq}|\phi_j\ast f(x)|^q\Big)^{p/q}dx\bigg)^{1/p}.
\end{equation}
We take obvious modification for $q=\infty$. Here $(\phi_j)_{j=0}^\infty$ is a carefully chosen family of Schwartz functions such that the convolution $\phi_j\ast f$ is defined on $\Omega$, see Definition \ref{Defn::Intro::Basic}.

In \cite[Proposition 6.6]{ShiYaoExt}, we used Rychkov's construction to prove that $\|f\|_{\Fs_{pq}^s(\Omega)}$ have equivalent (quasi-)norms via their derivatives. More precisely, let $m\ge1$, for every $0<p<\infty$, $0<q\le\infty$ and $s\in\R$ there is a $C=C(\Omega,p,q,s,m)>0$ such that
\begin{equation}\label{Eqn::Intro::IntroDerChar}
    C^{-1}\|f\|_{\Fs_{pq}^s(\Omega)}\le\sum_{|\alpha|\le m}\|\partial^\alpha f\|_{\Fs_{pq}^{s-m}(\Omega)}\le C\|f\|_{\Fs_{pq}^s(\Omega)},\quad\forall f\in\Fs_{pq}^s(\Omega). 
\end{equation}

Both \eqref{Eqn::Intro::IntroLPChar} and \eqref{Eqn::Intro::IntroDerChar} miss the endpoint: do we have the analogy of \eqref{Eqn::Intro::IntroLPChar} and \eqref{Eqn::Intro::IntroDerChar} for $p=\infty$? In this paper, we give the positive answers to both cases, by using the recently developed Triebel-Lizorkin-type spaces $\Fs_{pq}^{s\tau}$: we have the coincidences $\Fs_{\infty q}^s=\Fs_{pq}^{s,\frac1p}=\Bs_{qq}^{s,\frac1q}$ for $0<p<\infty$ (see \eqref{Eqn::Intro::CharFInftyQ}).

To make the results more general, we include the discussions of Besov-type spaces $\Bs_{pq}^{s\tau}$ and the Besov-Morrey spaces $\Ns_{pq}^{s\tau}$, see Definition \ref{Defn::Intro::DefMorrey}.

\medskip
We denote by $\Qc$ the set of dyadic cubes in $\R^n$, that is
\begin{equation}
    \Qc:=\{Q_{J,v}:J\in\Z,v\in\Z^n\},\quad\text{where}\quad Q_{J,v}:=2^{-J}v+(0,2^{-J})^n.
\end{equation}

Our result for \eqref{Eqn::Intro::IntroLPChar} is the following:

\begin{thm}[Littlewood-Paley type characterizations]\label{Thm::LPNorm} Let $\Omega=\{(x',x_n):x_n>\rho(x')\}\subset\R^n$ be a special Lipschitz domain and let $(\phi_j)_{j=0}^\infty$ be a Littlewood-Paley family associated with $\Omega$ (see Definition \ref{Defn::Intro::Basic}). Then for $0<p,q\le\infty$, $s\in\R$ and $\tau\ge0$ ($p<\infty$ for $\Fs$-cases), we have the following equivalent (quasi-)norms:
\begin{align*}
        \|f\|_{\Bs_{pq}^{s\tau}(\Omega)}&\approx_{\phi,p,q,s,\tau}\|(2^{js}\1_\Omega\cdot(\phi_j\ast f))_{j=0}^\infty\|_{\ell^qL^p_\tau}=
    \sup_{Q_{J,v}\in\Qc}2^{nJ\tau}\Big(\sum_{j=\max(0,J)}^\infty2^{jsq}\|\phi_j\ast f\|_{L^p(Q_{J,v}\cap\Omega)}^q\Big)^\frac1q;
    \\
    \|f\|_{\Fs_{pq}^{s\tau}(\Omega)}&\approx_{\phi,p,q,s,\tau}\|(2^{js}\1_\Omega\cdot(\phi_j\ast f))_{j=0}^\infty\|_{L^p_\tau\ell^q}=
    \sup_{Q_{J,v}\in\Qc}2^{nJ\tau}\Big(\int_{Q_{J,v}\cap\Omega}\Big(\sum_{j=\max(0,J)}^\infty2^{jsq}|\phi_j\ast f(x)|^q\Big)^\frac pqdx\Big)^\frac1p;
    \\
    \|f\|_{\Ns_{pq}^{s\tau}(\Omega)}&\approx_{\phi,p,q,s,\tau}\|(2^{js}\1_\Omega\cdot(\phi_j\ast f))_{j=0}^\infty\|_{\ell^qM^p_\tau}=\Big(\sum_{j=0}^\infty\sup_{Q_{J,v}\in\Qc} 2^{(js+nJ\tau)q}\|\phi_j\ast f\|_{L^p(Q_{J,v}\cap\Omega)}^q\Big)^\frac1q.
\end{align*}
 {\normalfont(See Definition \ref{Defn::Intro::SeqSpaces} for $\ell^qL^p_\tau$, $L^p_\tau\ell^q$ and $\ell^qM^p_\tau$.)}
 In particular for $0<q\le\infty$ and $s\in\R$,
$$\|f\|_{\Fs_{\infty q}^s(\Omega)}\approx_{\phi,q,s}\sup_{J\in\Z,v\in\Z^n}2^{J\frac nq}\int_{Q_{J,v}\cap\Omega}\Big(\sum_{j=\max(0,J)}^\infty2^{jqs}|\phi_j\ast f(x)|^qdx\Big)^\frac1q.$$
\end{thm}

One can also get some characterizations on bounded Lipschitz domain, whose expressions are less elegant however. See Remark \ref{Rmk::RmkBddDom}.

Similar to \cite[Theorem 2.3]{ExtensionLipschitz}, we also have the corresponding characterizations using Peetre maximal functions, see Proposition \ref{Prop::NormMax} and Corollary \ref{Cor::NormMaxP}.

\medskip Our result for \eqref{Eqn::Intro::IntroDerChar} is the following:
\begin{thm}[Equivalent norm characterizations via derivatives]\label{Thm::EqvNorm}
Let $\As\in\{\Bs,\Fs,\Ns\}$, $0<p,q\le\infty$, $s\in\R$ and $\tau\ge0$ ($p<\infty$ for $\Fs$-cases). Let $\Omega\subset\R^n$ be either a special Lipschitz domain or a bounded Lipschitz domain.
Then for any positive integer $m$, the space $\As_{pq}^{s\tau}(\Omega)$ has the following equivalent (quasi-)norm:
\begin{equation}\label{Eqn::Intro::EqvNorm}
    \|f\|_{\As_{p,q}^{s,\tau}(\Omega)}\approx_{p,q,s,m,\tau,\Omega}\sum_{|\alpha|\le m}\|\partial^\alpha f\|_{\As_{p,q}^{s-m,\tau}(\Omega)}.
\end{equation}
In particular $\|f\|_{\Fs_{\infty,q}^{s}(\Omega)}\approx_{q,s,m,\Omega}\sum_{|\alpha|\le m}\|\partial^\alpha f\|_{\Fs_{\infty,q}^{s-m}(\Omega)}$ for all $0<q\le\infty$ and $s\in\R$.
\end{thm}
The Besov-Morrey case $\As=\Ns$ of Theorem \ref{Thm::EqvNorm} was stated in \cite[Proposition 4.15]{YSYIntBesovMorrey}.  However, the key step in their proof requires \cite[(4.70)]{TriebelWavelet} (see \cite[Remark 4.14]{YSYIntBesovMorrey}), which cannot be achieved.
\begin{rem}\label{Rmk::RmkEqvNorm}
In the proof of \cite[Proposition 4.21]{TriebelWavelet}, Triebel claimed the following statement: 
\begin{equation}\label{Eqn::EqvNormRmk}
    \|f\|_{\As_{pq}^s(\Omega)}\approx\| Ef\|_{\As_{pq}^s(\R^n)}\approx\sum_{|\alpha|\le m}\|\partial^\alpha E f\|_{\As_{pq}^s(\R^n)}=\sum_{|\alpha|\le m}\|E\partial^\alpha  f\|_{\As_{pq}^s(\R^n)}\lesssim\sum_{|\alpha|\le m}\|\partial^\alpha  f\|_{\As_{pq}^s(\Omega)}.
\end{equation}
Here $E=E_\Omega$ is an extension operator which is bounded on $\As_{pq}^s(\Omega)\to\As_{pq}^s(\R^n)$ and $\As_{pq}^{s-m}(\Omega)\to\As_{pq}^{s-m}(\R^n)$.

However, the commutativity $\partial^\alpha\circ E=E\circ\partial^\alpha$ in \eqref{Eqn::EqvNormRmk} (see \cite[(4.70)]{TriebelWavelet}) cannot be achieved. In \cite[Section 1.2]{ShiYaoExt} we borrowed some facts from several complex variables to show that $\partial^\alpha\circ E=E\circ\partial^\alpha$ can never be true: if it is true (even locally) then $\overline{\partial}$-equation for $\Omega$ can gain 1 derivative.
To prove Theorem \ref{Thm::EqvNorm} (also to fix the proof of \cite[Proposition 4.15]{YSYIntBesovMorrey}), simply using the boundedness of $E_\Omega$ is not enough.

By observing \eqref{Eqn::EqvNormRmk} more carefully, the argument still works if $\partial^\alpha\circ E=E^\alpha\circ\partial^\alpha$ hold for some extension operators $E^\alpha:\As^{s-m}_{pq}(\Omega)\to\As^{s-m}_{pq}(\Omega)$. This can be done if $E$ is the standard half space extension\footnote{The half space extension works on $\R^n_+=\{x_n>0\}$. It has the form $Ef(x',x_n)=\sum_ja_jf(x',-b_jx_n)$ when $x_n<0$. In this case $E^\alpha f(x',x_n)=\sum_ja_j(-b_j)^{\alpha_n}f(x',-b_jx_n)$ has the similar expression to $E$.}. Using the operators $E^\alpha$ Triebel proved the equivalent norms via derivatives for $\R^n_+$ and for smooth domains, see \cite[Section 3.3.5]{TriebelTheoryOfFunctionSpacesI}.

In our case $E$ is Rychkov's extension operator (see \eqref{Eqn::ExtOp}). Even on special Lipschitz domain, it is not known to the author whether $\partial^\alpha\circ E=E^\alpha\circ\partial^\alpha$ can be achieved (which in general should have the form \eqref{Eqn::ThmBdd::DefT}). Nevertheless, a weaker form $\partial^\alpha\circ E=\sum_\beta E^{\alpha,\beta}\circ\partial^\beta$ is enough to fix \eqref{Eqn::EqvNormRmk}. In the proof we introduce $E^{\alpha,\beta}$ in \eqref{Eqn::EAlphaBeta} and get the proof using \eqref{Eqn::EqvNormPfKey}.

See also \cite[Section 2.2 and Remark 6.5]{ShiYaoExt}.
\end{rem}

\section{Function Spaces and Notations}

Let $U\subseteq\R^n$ be an open set, we define $\Ss'(U)$ to be the space of restricted tempered distributions:\\ $\Ss'(U):=\{\tilde f|_U:\tilde f\in\Ss'(\R^n)\}$. See also \cite[Proposition 3.1]{ExtensionLipschitz}.

We use the notation $A \lesssim B$ to mean that $A \leq CB$ where $C$ is a constant independent of $A,B$. We use $A \approx B$ for ``$A \lesssim B$ and $B \lesssim A$''. And we use $A\lesssim_xB$ to emphasize that the constant depends on the quantity $x$.

When $p$ or $q<1$, we use ``norms'' (for  $\As_{pq}^{s\tau}$ etc.) as the abbreviation to the usual ``quasi-norms''. 

\medskip
In the paper we use the following Littlewood-Paley family, whose elements do not have compact supports in the Fourier side. It is crucially useful in the construction of Rychkov's extension operator.
\begin{defn}\label{Defn::Intro::Basic}
Let $\Omega=\{x_n>\rho(x')\}$ be a special Lipschitz domain, a \textbf{Littlewood-Paley family} associated with $\Omega$ is a sequence $\phi=(\phi_j)_{j=0}^\infty\subset\Ss(\R^n)$ of Schwartz functions that satisfies the following:
\begin{enumerate}[label=(P.\alph*)]
\item\label{Item::LPCond::Momt} \textit{Moment condition}: $\int x^\alpha\phi_1(x)dx=0$ for all multi-indices $\alpha\in\Z_{\ge0}^n$.
    \item\label{Item::LPCond::Scal} \textit{Scaling condition}: $\phi_j(x)=2^{(j-1)n}\phi_1(2^{j-1}x)$ for all $j\ge2$.
	
	\item\label{Item::LPCond::AppId} \textit{Approximate identity}: $\sum_{j=0}^\infty\phi_j=\delta_0$ is the Dirac delta measure.
	\item\label{Item::LPCond::Supp} \textit{Support condition}: $\supp\phi_j\subset\{(x',x_n): x_n<-\|\nabla\rho\|_{L^\infty}\cdot|x'|\}$ for all $j\ge0$.
\end{enumerate}
\end{defn}

In the paper we use the sequence spaces $\ell^qL^p_\tau$, $L^p_\tau \ell^q$, $\ell^qM^p_\tau$ given by the following: 

\begin{defn}\label{Defn::Intro::SeqSpaces}
Let $0<p,q\le\infty$ and $\tau\ge0$. We denote by $\ell^qL^p_\tau(\R^n)$ and $L^p_\tau\ell^q(\R^n)$ the spaces of vector valued measurable functions $(f_j)_{j=0}^\infty\subset L^p_\loc(\R^n)$ such that the following (quasi-)norms are finite respectively:
\begin{align}\notag
    \|(f_j)_{j=0}^\infty\|_{\ell^qL^p_\tau}&:=\sup_{Q_{J,v}\in\Qc}2^{nJ\tau}\|(f_j)_{j=\max(0,J)}^\infty\|_{\ell^q(L^p(Q_{J,v}))}=\sup_{J\in\Z,v\in\Z^n}2^{nJ\tau}\Big(\sum_{j=\max(0,J)}^\infty\|f_j\|_{L^p(Q_{J,v})}^q\Big)^\frac1q;
    \\\notag
    \|(f_j)_{j=0}^\infty\|_{L^p_\tau\ell^q}&:=\sup_{Q_{J,v}\in\Qc}2^{nJ\tau}\|(f_j)_{j=\max(0,J)}^\infty\|_{L^p(Q_{J,v};\ell^q)}=\sup_{J\in\Z,v\in\Z^n}2^{nJ\tau}\bigg(\int_{Q_{J,v}}\Big(\sum_{j=\max(0,J)}^\infty|f_j(x)|^q\Big)^\frac pqdx\bigg)^\frac1p.
    % \\\label{Eqn::Intro::NormVecNs}
    % \|F\|_{\ell^q_\tau L^p}&:=\big\|\big(\sup_{Q_{J,v}\in\Qc;J\le j}2^{nJ\tau}\|f_j\|_{L^p(Q_{J,v})}\big)_{j=0}^\infty\big\|_{\ell^q}=\Big(\sum_{j=0}^\infty\sup_{v\in\Z^n,J\le j}2^{nJ\tau q}\|f_j\|_{L^p(Q_{J,v})}^q\Big)^\frac1q.
\end{align}

We define the Morrey space\footnote{Our notation is different from the standard one, which can be found in for example \cite[Definition 2.1]{TangXuMorrey}.} $M^p_\tau(\R^n)$ to be the set of all $f\in L^p_\loc(\R^n)$ whose (quasi-)norm below is finite:
\begin{equation*}
    \|f\|_{M^p_\tau}:=\textstyle\sup_{Q_{J,v}\in\Qc}2^{nJ\tau}\|f\|_{L^p(Q_{J,v})}.
\end{equation*}

We define $\ell^qM^p_\tau(\R^n):=\ell^q(\Z_{\ge0};M^p_\tau(\R^n))$ with $\|(f_j)_{j=0}^\infty\|_{\ell^qM^p_\tau}:=\big(\sum_{j=0}^\infty\|f_j\|_{M^p_\tau(\R^n)}^q\big)^\frac1q$.
\end{defn}

Our Besov-type spaces $\Bs_{pq}^{s\tau}$, Triebel-Lizorkin-type spaces $\Fs_{pq}^{s\tau}$ and Besov-Morrey spaces $\Ns_{pq}^{s\tau}$ are given by the following:
\begin{defn}\label{Defn::Intro::DefMorrey}
Let $\lambda=(\lambda_j)_{j=0}^\infty$ be a sequence of Schwartz functions satisfying:
\begin{enumerate}[label=(P.\alph*')]
    \item\label{Item::CLPCond::Four} The Fourier transform $\hat\lambda_0(\xi)=\int_{\R^n}\lambda_0(x)2^{-2\pi ix\xi}dx$ satisfies $\supp\hat\lambda_0\subset\{|\xi|<2\}$ and $\hat\lambda_0|_{\{|\xi|<1\}}\equiv1$.
    \item\label{Item::CLPCond::Scal}  $\lambda_j(x)=2^{jn}\lambda_0(2^jx)-2^{(j-1)n}\lambda_0(2^{j-1}x)$ for $j\ge1$.
\end{enumerate}

Let $0<p,q\le\infty$, $s\in\R$ and $\tau\ge0$ ($p<\infty$ for $\Fs$-cases). We define the \textit{Besov-type Morrey space} $\Bs_{pq}^{s\tau}(\R^n)$, the \textit{Triebel-Lizorkin-type Morrey space} $\Fs_{pq}^{s\tau}(\R^n)$ and the \textit{Besov-Morrey space} $\Ns_{pq}^{s\tau}(\R^n)$, to be the sets of all tempered distributions $f\in\Ss'(\R^n)$ such that the following norms are finite, respectively:
\begin{gather}\label{Eqn::DefMorrey::MrNorm}
\|f\|_{\Bs_{pq}^{s\tau}(\R^n)}:=\|(2^{js}\lambda_j\ast f)_{j=0}^\infty\|_{\ell^qL^p_\tau};\ \|f\|_{\Fs_{pq}^{s\tau}}:=\|(2^{js}\lambda_j\ast f)_j\|_{L^p_\tau\ell^q};\
\|f\|_{\Ns_{pq}^{s\tau}}:=\|(2^{js}\lambda_j\ast f)_j\|_{\ell^qM^p_\tau}.
\end{gather}

Let $\As\in\{\Bs,\Fs,\Ns\}$. For an (arbitrary) open subset $U\subseteq\R^n$, we define $\As_{pq}^{s\tau}(U):=\{\tilde f|_U:\tilde f\in\As_{pq}^{s\tau}(\R^n)\}$ ($p<\infty$ for $\Fs$-cases) with the norm
\begin{equation}\label{Eqn::Intro::NormDomain}
    \textstyle\|f\|_{\As_{pq}^{s\tau}(U)}:=\inf\{\|\tilde f\|_{\As_{pq}^{s\tau}(\R^n)}:\tilde f\in\As_{pq}^{s\tau}(\R^n), \tilde f|_U=f\}.
\end{equation}
\end{defn}

The definitions of the spaces $\As_{pq}^{s\tau}(U)$ do not depend on the choice of $(\lambda_j)_{j=0}^\infty$ which satisfies  \ref{Item::CLPCond::Four} and \ref{Item::CLPCond::Scal}. See \cite[Page 39, Corollary 2.1]{YSYMorrey} and \cite[Theorem 2.8]{TangXuMorrey}. %The conditions for $(\lambda_j)_j$ can be weaken (see Section \ref{Sec::Rmk}).

\begin{rem}\label{Rmk::Spaces}
We remark some known results and different notations for these spaces in $\R^n$ from the literature:

\begin{enumerate}[(i)]
    \item Clearly $\Bs_{pq}^s(\R^n)=\Bs_{pq}^{s0}(\R^n)=\Ns_{pq}^{s0}(\R^n)$ and $\Fs_{pq}^{s0}(\R^n)=\Fs_{pq}^s(\R^n)$ (provided $p<\infty$).
    
    \item In applications only $0\le\tau\le\frac1p$ is interesting: by \cite[Theorem 2]{YangYuanMorreyEquiv} and \cite[Lemma 3.4]{SickelMorrey1},
\begin{equation}\label{Eqn::Intro::EqvSpaces}
    \Bs_{p,q}^{s,\tau}(\R^n)=\Fs_{p,q}^{s,\tau}(\R^n)=\Bs_{\infty,\infty}^{s+n(\tau-\frac1p)}(\R^n),\quad\Ns_{p,q}^{s,\tau}(\R^n)=\{0\},\quad\forall\,0<p,q\le\infty,\ s\in\R,\ \tau>\tfrac1p.
\end{equation}

\item For the case $\tau=1/p$, by \cite[Theorem 2]{YangYuanMorreyEquiv} and \cite[Remark 11(ii)]{SickelMorrey1},
\begin{equation*}%\label{Eqn::Intro::EqvSpaces2}
    \Bs_{p,\infty}^{s,\frac1p}(\R^n)=\Fs_{p,\infty}^{s,\frac1p}(\R^n)=\Bs_{\infty,\infty}^{s}(\R^n),\quad\Ns_{p,q}^{s,\frac1p}(\R^n)=\Bs_{\infty,q}^s(\R^n),\quad\forall\,0<p,q\le\infty,\ s\in\R.
\end{equation*}

\item Although $\Fs_{pq}^{s\tau}$-spaces are only defined for $p<\infty$, we have a description for $\Fs_{\infty q}^s$-spaces as the following (see \cite[Page 41, Proposition 2.4(iii)]{YSYMorrey} and \cite[Section 5]{FrazierJawerthFInftyQ}):
\begin{equation}\label{Eqn::Intro::CharFInftyQ}
    \Fs_{\infty q}^s(\R^n)=\Fs_{p,q}^{s,\frac1p}(\R^n)=\Bs_{q,q}^{s,\frac1q}(\R^n),\quad\forall\,0<p<\infty,\ 0<q\le\infty,\ s\in\R.
\end{equation}

\item Our notation $\Ns_{pq}^{s\tau}$ corresponds to the $\mathcal B_{pq}^{s\tau}$ in \cite[Definition 5]{SickelMorrey1}. For the classical notations\footnote{Some papers may have different order of the indices. For example, in \cite{MazzucatoBesovMorrey} this is written as $\Nc_{upq}^s$.} $\Nc_{uqp}^s$ we have correspondence (see  \cite[Remark 13(iii)]{SickelMorrey1} for example):
\begin{equation*}%\label{Eqn::Intro::CharNc}
    \Nc_{u,q,p}^s(\R^n)=\Ns_{p,q}^{s,\frac1p-\frac1u}(\R^n),\quad\forall\,0<p\le u\le\infty,\ 0<q\le\infty,\ s\in\R.
\end{equation*}

\item We do not talk about the \textit{Triebel-Lizorkin-Morrey spaces} $\Ec_{uqp}^s$ in the paper, because they are special cases of the Triebel-Lizorkin-type spaces: we have $\Ec_{u,q,p}^s(\R^n)=\Fs_{p,q}^{s,1/p-1/u}(\R^n)$ for all $p\in(0,\infty)$, $q\in(0,\infty]$, $u\in[p,\infty]$ and $s\in\R$. See \cite[Corollary 3.3]{YSYMorrey}.

\item There are also papers that use the notations $\Lambda^\varrho\As_{pq}^s$ and $\Lambda_{\varrho}\As_{pq}^s$ for $\As\in\{\Bs,\Fs\}$ and $-n\le\varrho\le0$ ($p<\infty$ for $\Fs$-cases), for example \cite{TriebelTheoryOfFunctionSpacesIV,HaroskeTriebelMorrey}. These spaces describe the same collection to $\As_{pq}^{s\tau}$ for $\As\in\{\Bs,\Fs,\Ns\}$, see \cite[Remarks 2.7 and 2.9]{HaroskeTriebelMorrey} for example.

\end{enumerate}

For more discussions, we refer the reader to \cite{YSYMorrey,TriebelHybrid,HaroskeTriebelMorrey}.
\end{rem}

\section{Proof of the Theorems}

Our proof follows from some results in \cite{ExtensionLipschitz} and \cite{YangYuan10}. 

The key ingredient is the \textbf{Peetre maximal operators} introduced in \cite{Peetre}.
\begin{defn}
Let $N>0$, $U\subseteq\R^n$ be an open set and let $\eta=(\eta_j)_{j=0}^\infty$ be a sequence of Schwartz functions. The associated \textit{Peetre maximal operators} $(\Pc^{\eta,N}_{U,j})_{j=0}^\infty$ are given by
\begin{equation*}% \label{Eqn::PeeMaxDef}
    \Pc^{\eta,N}_{U,j}f(x):=\sup\limits_{y\in U}\frac{|\eta_j\ast f(y)|}{(1+2^j|x-y|)^{N}},\quad f\in\Ss'(\R^n),\quad x\in\R^n,\quad j\ge0.
\end{equation*} 
% When $U=\R^n$ we use abbreviation $\Pc^{\eta,N}_j:=\Pc^{\eta,N}_{\R^n,j}$.
\end{defn}

% Recall the following lemmas from \cite{ExtensionLipschitz} and \cite{YangYuan10}.
\begin{lemma}\label{Lem::ExistPsi}
Let $\phi=(\phi_j)_{j=0}^\infty$ be a Littlewood-Paley family associated with a special Lipschitz domain $\Omega$ (see Definition \ref{Defn::Intro::Basic}). Then there is a $\psi=(\psi_j)_{j=0}^\infty\subset\Ss'(\R^n)$ satisfying \ref{Item::LPCond::Momt} and \ref{Item::LPCond::Scal} such that $(\psi_j\ast \phi_j)_{j=0}^\infty$ is also associated with $\Omega$.
\end{lemma}
\begin{proof}The assumptions $\phi_j(x)=2^{(j-1)n}\phi_1(2^{j-1}x)$ for $j\ge1$ and $\sum_{j=0}^\infty\phi_j=\delta_0$ imply $\phi_1(x)=2^n\phi_0(2x)-\phi_0(x)$, i.e. $\hat\phi_1(\xi)=\hat\phi_0(\xi/2)-\hat\phi_0(\xi)$.
    We can take $\psi=(\psi_j)_{j=0}^\infty$ via the Fourier transforms: 
    $$ \hat\psi_0(\xi):=2\hat \phi_0(\xi)-\hat\phi_0(\xi)^3;\qquad\hat\psi_j(\xi):=(\hat\phi_0(2^{-j}\xi)+\hat\phi_0(2^{1-j}\xi))(2-\hat\phi_0(2^{-j}\xi)^2-\hat\phi_0(2^{1-j}\xi)^2),\text{ for }j\ge1.$$ See \cite[Proposition 2.1]{ExtensionLipschitz} for details.
\end{proof}

\begin{lemma}[{\cite[Lemma 2.1]{Bui}}]\label{Lem::Heideman}
Let $\eta=(\eta_j)_{j=0}^\infty$ and $\theta=(\theta_j)_{j=0}^\infty\subset\Ss(\R^n)$ both satisfy conditions \ref{Item::LPCond::Momt} and \ref{Item::LPCond::Scal}. Then for any $N>0$ there exists a $C=C(\eta,\theta,N)>0$ such that\begin{equation*}%\label{Eqn::Heideman1}
    \int_{\R^n}|\eta_j\ast\theta_k(x)|(1+2^k|x|)^Ndx\lesssim_{\eta,\theta,N}2^{-N|j-k|},\qquad\forall j,k\ge0.
\end{equation*}
\end{lemma}

\begin{lemma}\label{Lem::ASumLem}
Let $0<p,q\le\infty$, $\tau\ge0$ and $\delta>n\tau$. There is a $C=C(n,p,q,\tau,\delta)>0$ such that for every $(g_j)_{j=0}^\infty\subset L^p_\loc(\R^n)$,
\begin{align}\label{Eqn::ASumLem::Bs}
    \Big\|\Big(\sum_{k\ge0}2^{-\delta|j-k|}g_k\Big)_{j=0}^\infty\Big\|_{\ell^qL^p_\tau}&\le C\|(g_j)_{j=0}^\infty\|_{\ell^qL^p_\tau};
    \\\label{Eqn::ASumLem::Fs}
    \Big\|\Big(\sum_{k\ge0}2^{-\delta|j-k|}g_k\Big)_{j=0}^\infty\Big\|_{L^p_\tau\ell^q}&\le C\|(g_j)_{j=0}^\infty\|_{L^p_\tau\ell^q},\qquad\text{provided }p<\infty;
     \\\label{Eqn::ASumLem::Ns}
    \Big\|\Big(\sum_{k\ge0}2^{-\delta|j-k|}g_k\Big)_{j=0}^\infty\Big\|_{\ell^qM^p_\tau}&\le C\|(g_j)_{j=0}^\infty\|_{\ell^q M^p_\tau}.
\end{align}
\end{lemma}
 \begin{proof}
 \eqref{Eqn::ASumLem::Bs} and \eqref{Eqn::ASumLem::Fs} have been done in \cite[Lemma 2.3]{YangYuan10}. We only prove \eqref{Eqn::ASumLem::Ns}.
 
Using the case $\tau=0$ in \eqref{Eqn::ASumLem::Bs} we have $$\textstyle\big\|\big(\sum_{k\ge0}2^{-\delta|j-k|}f_k\big)_{j=0}^\infty\big\|_{\ell^q(L^p)}\lesssim_{p,q,\delta}\|(f_j)_{j=0}^\infty\|_{\ell^q(L^p)},\quad\forall (f_j)_{j=0}^\infty\in \ell^q(\Z_{\ge0};L^p(\R^n)).$$
 
 Note that $\|g_k\|_{M^p_\tau}=\|\sup_{Q_{J,v}}|2^{nJ\tau}\1_{Q_{J,v}}\cdot g_k|\|_{L^p(\R^n)}$. By taking $f_k:=\sup_{Q_{J,v}}|2^{nJ\tau}\1_{Q_{J,v}}\cdot g_k|$ above we have
 \begin{align*}
     &\Big\|\Big(\sum_{k\ge0}2^{-\delta|j-k|}|g_k|\Big)_{j=0}^\infty\Big\|_{\ell^qM^p_\tau}=\Big\|\Big(\sup_{Q_{J,v}\in\Qc}2^{nJ\tau}\1_{Q_{J,v}}\cdot\sum_{k\ge0}2^{-\delta|j-k|}|g_k|\Big)_{j=0}^\infty\Big\|_{\ell^q(L^p)}
     \\
     &\quad\le\Big\|\Big(\sum_{k\ge0}2^{-\delta|j-k|}\sup_{Q_{J,v}\in\Qc}2^{nJ\tau}\1_{Q_{J,v}}\cdot |g_k|\Big)_{j=0}^\infty\Big\|_{\ell^q(L^p)}=\Big\|\Big(\sum_{k\ge0}2^{-\delta|j-k|}f_k\Big)_{j=0}^\infty\Big\|_{\ell^q(L^p)}
     \\
     \hspace{0.8in}&\quad\lesssim_{p,q,\delta}\|(f_j)_{j=0}^\infty\|_{\ell^q(L^p)}=\|(g_j)_{j=0}^\infty\|_{\ell^qM^p_\tau}.&\hspace{0.8in}\qedhere
 \end{align*}
\end{proof}

\begin{lemma}\label{Lem::STInq}
Let $\Omega\subset\R^n$ be a special Lipschitz domain,  let $\phi=(\phi_j)_{j=0}^\infty$ be a Littlewood-Paley family associated with $\Omega$, and let $\theta=(\theta_j)_{j=0}^\infty$ satisfies conditions \ref{Item::LPCond::Momt}, \ref{Item::LPCond::Scal} and \ref{Item::LPCond::Supp}. Then for any $N>0$ and $\gamma\in(0,\infty]$ there is a $C=C(\theta,\phi,N)>0$, such that,
\begin{equation}\label{Eqn::STInq::omega}
    \Pc^{\theta,N}_{\Omega,j}f(x)\le C\bigg(\sum_{k=0}^\infty 2^{-N\gamma|j-k|}\int_{\Omega}\frac{2^{kn}|\phi_k\ast f(y)|^\gamma dy}{(1+2^k|x-y|)^{N\gamma}}\bigg)^{1/\gamma},\quad \forall f\in\Ss'(\R^n),\ j\ge0,\ x\in \Omega.
\end{equation}
\end{lemma}

\begin{proof}
The special case $\theta=\phi$ of \eqref{Eqn::STInq::omega} is proved in \cite[Proof of Theorem 3.2, Step 1]{ExtensionLipschitz}. Namely, we have
\begin{equation}\label{Eqn::STInq::Special}
    \Pc^{\phi,N}_{\Omega,j}f(x)\lesssim_{\phi,N}\bigg(\sum_{k=0}^\infty 2^{-N\gamma|j-k|}\int_{\Omega}\frac{2^{kn}|\phi_k\ast f(y)|^\gamma dy}{(1+2^k|x-y|)^{N\gamma}}\bigg)^{1/\gamma},\quad \forall f\in\Ss'(\R^n),\ j\ge0,\ x\in \Omega.
\end{equation}
Also see \cite[Proof of Theorem 2.6, Step 1]{UllrichCoorbit} for the argument. Thus it suffices to prove the case $\gamma=\infty$:
\begin{equation}\label{Eqn::STInq::Gamma=Infty}
    \Pc^{\theta,N}_{\Omega,j}f(x)\lesssim_{\theta,\phi,N}\sup_{k\ge0}2^{-N|j-k|}\Pc^{\phi,N}_{\Omega,k}f(x),\quad\forall f\in\Ss'(\R^n),\quad j\ge0,\quad x\in\Omega.
\end{equation}

Let $\psi=(\psi_j)_{j=0}^\infty$ satisfies the consequence of Lemma \ref{Lem::ExistPsi}, so $\theta_j\ast f=\sum_{k=0}^\infty(\theta_j\ast\psi_k)\ast(\phi_k\ast f)$ for $j\ge0$. 
By assumption $\phi_j,\psi_j,\theta_j$ are supported in $K=\{x_n<-\|\nabla\rho\|_{L^\infty}\cdot|x'|\}$ where $\rho$ is the defining function for $\Omega=\{x_n>\rho(x')\}$. Using the property $\Omega-K\subseteq\Omega$, we have
\begin{align*}\1_\Omega\cdot(\theta_j\ast f)&\textstyle=\1_\Omega\cdot\sum_{k=0}^\infty(\theta_j\ast\psi_k)\ast(\1_\Omega\cdot(\phi_k\ast f));
\\
    \text{and thus}\quad \Pc^{\theta,N}_{\Omega,j}f(x)&=\sup_{z\in\Omega}\frac{|\theta_j\ast f(z)|}{(1+2^j|x-z|)^N}\le\sup_{z\in\Omega}\sum_{k=0}^\infty\int_\Omega\frac{|\theta_j\ast\psi_k(z-y)||\phi_k\ast f(y)|dy}{(1+2^j|x-z|)^N}.
\end{align*}

The elementary inequality yields
\begin{equation*}\label{Eqn::STInq::TmpEle}
    \frac{1}{(1+2^j|x-z|)^N}\le\frac{2^{N|j-k|}}{(1+2^k|x-z|)^N}\frac{(1+2^k|z-y|)^N}{(1+2^k|z-y|)^N}\le 2^{N|j-k|}\frac{(1+2^k|z-y|)^N}{(1+2^k|x-y|)^N}.
\end{equation*}

Therefore, 
\begin{equation}\label{Eqn::STInq::PfGamma=Infty}
    \begin{aligned}
    \Pc^{\theta,N}_{\Omega,j}f(x)&=\sup_{z\in\Omega}\frac{|\phi_k\ast f(z)|}{(1+2^k|x-z|)^N}\sum_{k=0}^\infty\int_\Omega 2^{N|j-k|}|\theta_j\ast\psi_k(z-y)|(1+2^k|z-y|)^Ndy
    \\
    &\le\sup_{k\ge0}2^{-N|j-k|}\Pc_{\Omega,k}^{\phi,N}f(x)\sum_{l=0}^\infty \int_\Omega 2^{2N|j-l|}|\theta_j\ast\psi_l(y)|(1+2^l|y|)^Ndy
    \\
    &\lesssim_{\theta,\phi,N}\sup_{k\ge0}2^{-N|j-k|}\Pc_{\Omega,k}^{\phi,N}f(x)\sum_{l=0}^\infty 2^{(2N-(2N+1))|j-l|}\lesssim\sup_{k\ge0}2^{-N|j-k|}\Pc_{\Omega,k}^{\phi,N}f(x).
\end{aligned}
\end{equation}
Here the last inequality is obtained by applying Lemma \ref{Lem::Heideman}.

Therefore we get \eqref{Eqn::STInq::Gamma=Infty}. Combining it with \eqref{Eqn::STInq::Special} we complete the proof.
\end{proof}

Recall the Hardy-Littlewood maximal function $\M f(x):=\sup_{R>0}|B(0,R)|^{-1}\int_{B(x,R)}|f(y)|dy$ for $f\in L^1_\loc$.
\begin{lemma}\label{Lem::PeeToHL}
Let $N>n$. There is a $C=C(N)>0$ such that for any $g\in L^1_\loc(\R^n)$,
\begin{equation}\label{Eqn::PeeToHL}
    \int_{\R^n}\frac{2^{kn}|g(y)|dy}{(1+2^k|x-y|)^N}\le C\sum_{w\in\Z^n}\frac1{(1+|v-w|)^{N-n}}\cdot \M(\1_{Q_{J,w}}\cdot g)(x),\quad J\in\Z,\ v\in\Z^n,\ k\ge J,\ x\in Q_{J,v}.
\end{equation}
\end{lemma}

Our lemma here is weaker than the corresponding estimate in \cite[Proof of Theorem 1.2, Step 3]{YangYuan10}.
\begin{proof}
By taking a translation, it suffices to prove the estimate on $x\in Q_{J,0}$, i.e for $v=0$. Note that if $y\in Q_{J,w}$, then $|x-y|\ge\dist (Q_{J,w},Q_{J,0})\ge \frac1{\sqrt n}2^{-J}\max(0,|w|-\sqrt n)$ and $|x-y|\le |w|+\sqrt n$. Therefore
\begin{align*}
    &\int_{\R^n}\frac{2^{kn}|g(y)| dy}{(1+2^k|x-y|)^N}\le\int_{B(x,3\sqrt n 2^{-J})}\frac{2^{kn}|g(y)|  dy}{(1+2^k|x-y|)^N}+\sum_{|w|>2\sqrt n}\int_{Q_{J,w}}\frac{2^{kn}|g(y)|  dy}{(1+2^k|x-y|)^N}
    \\
    &\quad\lesssim\Big\|\frac{2^{n(k-J)}}{(1+2^k|y|)^N}\Big\|_{L^1(\R^n_y)}\M(\1_{B(0,4\sqrt n 2^{-J})}\cdot  g)(x)+\sum_{|w|>2\sqrt n}\frac{2^{kn}}{\big(1+2^k 2^{-J}(\frac{|w|}{\sqrt n}-1)\big)^N}\int_{Q_{J,w}}|g(y)| dy
    \\
    &\quad\lesssim\sum_{|w|<4\sqrt n}\M(\1_{Q_{J,w}}\cdot g)(x)+\sum_{|w|>2\sqrt n}\frac{2^{-(k-J)(N-n)}}{|w|^{N-n}}\cdot\frac{2^{nJ}}{|w|^n}\int_{B(x,2^{-J}(|w|+\sqrt n))}|\1_{Q_{J,w}}\cdot g(y)| dy
    \\
    \hspace{0.3in}&\quad\lesssim\sum_{w\in\Z^n}\frac1{(1+|w|)^{N-n}}\cdot \M(\1_{Q_{J,w}}\cdot g)(x).&\hspace{0.3in}\qedhere
\end{align*}
\end{proof}
Combining Lemmas \ref{Lem::ASumLem} - \ref{Lem::PeeToHL} we have the following Morrey-type estimates for Peetre maximal functions.
\begin{prop}\label{Prop::PeeEst}Keeping the assumptions of Lemma \ref{Lem::STInq}, for every $0<p,q\le\infty$, $s\in\R$, $\tau\ge 0$ and $N>~\!\max(2n/\min(p,q),|s|+n\tau)$, there is a $C=C(\theta,\phi,p,q,s,\tau,N)>0$ such that for every $f\in\Ss'(\Omega)$,
\begin{align}\label{Eqn::PeeEst::SpeBs}
    \big\|\big(2^{js}\1_\Omega\cdot(\Pc^{\theta,N}_{\Omega,j}f)\big)_{j=0}^\infty\big\|_{\ell^qL^p_\tau}
    &\le C\big\|\big(2^{js}\1_\Omega\cdot(\phi_j\ast f)\big)_{j=0}^\infty\big\|_{\ell^qL^p_\tau};
    \\\label{Eqn::PeeEst::SpeFs}
    \big\|\big(2^{js}\1_\Omega\cdot(\Pc^{\theta,N}_{\Omega,j}f)\big)_{j=0}^\infty\big\|_{L^p_\tau\ell^q}
    &\le C\big\|\big(2^{js}\1_\Omega\cdot(\phi_j\ast f)\big)_{j=0}^\infty\big\|_{L^p_\tau\ell^q},\qquad\text{provided }p<\infty;
    \\\label{Eqn::PeeEst::SpeNs}
    \big\|\big(2^{js}\1_\Omega\cdot(\Pc^{\theta,N}_{\Omega,j}f)\big)_{j=0}^\infty\big\|_{\ell^qM^p_\tau}
    &\le C\big\|\big(2^{js}\1_\Omega\cdot(\phi_j\ast f)\big)_{j=0}^\infty\big\|_{\ell^qM^p_\tau}.
\end{align}
\end{prop}

\begin{rem}\label{Rmk::RmkN}
It is possible that the assumption $N>\max(\frac{2n}{\min(p,q)},|s|+n\tau)$ can be relaxed to $N>\frac n{\min(p,q)}$. In applications, we only need a large enough $N$ that does not depend on $f$.

A similar result for \eqref{Eqn::PeeEst::SpeNs} can be found in \cite[Proposition 2.12]{TangXuMorrey}. Note that we require $\theta_j$ to have Fourier compact supports in that proposition.
\end{rem}
\begin{proof}We use a convention $\phi_j:\equiv0$ for $j\le-1$. Thus in the computations below every sequence $(a_j)_{j=J}^\infty$ is identical to $(a_j)_{j=\max(0,J)}^\infty$.

By the assumption on $N$ we can take $\gamma\in(0,\min(p,q))$ such that $N\gamma>2n$. We first prove \eqref{Eqn::PeeEst::SpeFs}. 

Since $N>|s|+n\tau$. By Lemma  \ref{Lem::STInq} and using $2^{j\gamma s}2^{-N\gamma|j-k|}\le2^{-(N-|s|)\gamma|j-k|}2^{k\gamma s} $,
\begin{equation*}
    \|(2^{js}\Pc^{\theta,N}_{\Omega,j}f)_{j=0}^\infty\|_{L^p_\tau\ell^q}
    =\big\|\big(2^{j\gamma s}(\Pc^{\theta,N}_{\Omega,j}f)^\gamma\big)_{j=0}^\infty\big\|_{L^{\frac p\gamma}_{\tau\gamma}\ell^\frac q\gamma}^\frac1\gamma
    \lesssim\bigg\|\Big(\sum_{k=0}^\infty 2^{(|s|-N)\gamma|j-k|}\int_\Omega\frac{2^{kn}|2^{ks}\phi_k\ast f(y)|^\gamma dy}{(1+2^k|\cdot -y|)^{N\gamma}}\Big)_{j=0}^\infty\bigg\|_{L^{\frac p\gamma}_{\tau\gamma}\ell^\frac q\gamma}^\frac1\gamma.
\end{equation*}
By Lemma \ref{Lem::ASumLem} and since $(N-|s|)\gamma>n\tau\gamma$,
\begin{equation*}
    \bigg\|\Big(\sum_{k=0}^\infty 2^{(|s|-N)\gamma|j-k|}\int_\Omega\frac{2^{kn}|2^{ks}\phi_k\ast f(y)|^\gamma dy}{(1+2^k|\cdot -y|)^{N\gamma}}\Big)_{j=0}^\infty\bigg\|_{L^{\frac p\gamma}_{\tau\gamma}\ell^\frac q\gamma}
    \lesssim \bigg\|\Big(\int_\Omega\frac{2^{kn}|2^{ks}\phi_k\ast f(y)|^\gamma dy}{(1+2^k|\cdot -y|)^{N\gamma}}\Big)_{k=0}^\infty\bigg\|_{L^{\frac p\gamma}_{\tau\gamma}\ell^\frac q\gamma}.
\end{equation*}
Applying Lemma \ref{Lem::PeeToHL} with $g(x)=\1_\Omega(x)\cdot|2^{ks}\phi_k\ast f(x)|^\gamma$ for each $k\ge0$ and expanding the $L^{\frac p\gamma}_{\tau\gamma}\ell^\frac q\gamma$-norm,
\begin{align*}
    &\bigg\|\Big(\int_\Omega\frac{2^{kn}|2^{ks}\phi_k\ast f(y)|^\gamma dy}{(1+2^k|\cdot -y|)^{N\gamma}}\Big)_{k=0}^\infty\bigg\|_{L^{\frac p\gamma}_{\tau\gamma}\ell^\frac q\gamma}=\sup_{J\in\Z,v\in\Z^n}2^{nJ\tau\gamma\cdot\frac1\gamma}\bigg\|\Big(\int_\Omega\frac{2^{kn}|2^{ks}\phi_k\ast f(y)|^\gamma dy}{(1+2^k|\cdot-y|)^{N\gamma}}\Big)_{k=J}^\infty\bigg\|_{L^\frac p\gamma(Q_{J,v};\ell^\frac q\gamma)}^\frac1\gamma
        \\
        &\quad\lesssim_{N,\gamma}\sup_{J\in\Z,v\in\Z^n}2^{nJ\tau}\bigg\|\Big(\sum_{w\in\Z^n}\frac{1}{(1+|w-v|)^{N\gamma-n}}\M(\1_{Q_{J,w}}\cdot\1_\Omega\cdot|2^{ks}\phi_k\ast f|^\gamma)\Big)_{k=J}^\infty\bigg\|_{L^\frac p\gamma(Q_{J,v};\ell^\frac q\gamma)}^\frac1\gamma
        \\
        &\quad\le\Big(\sum_{v\in\Z^n}\frac1{(1+|v|)^{N\gamma-n}}\Big)^{1/\gamma}\sup_{J\in\Z,w\in\Z^n}2^{nJ\tau}\big\|\big(\M(\1_{Q_{J,w}\cap\Omega}\cdot|2^{ks}\phi_k\ast f|^\gamma)\big)_{k=J}^\infty\big\|_{L^\frac p\gamma(\R^n;\ell^\frac q\gamma)}^\frac1\gamma.
\end{align*}
Since $N\gamma-n>n$ the sum $\sum_{v\in\Z^n}(1+|v|)^{n-N\gamma}$ is finite. 

Finally, applying Fefferman-Stein's inequality to $\big(\M(\1_{Q_{J,w}\cap\Omega}\cdot|2^{ks}\phi_k\ast f|^\gamma)\big)_{k=J}^\infty$ in $L^\frac p\gamma(\R^n;\ell^\frac q\gamma)$ for each $J\in\Z$ (see \cite[Theorem 1(1)]{FeffermanSteinVectHL} and also \cite[Remark 5.6.7]{GrafakosClassical}), since $1<p/\gamma<\infty$ and $1< q/\gamma\le\infty$,
\begin{align*}
    &\sup_{Q_{J,w}\in\Qc}2^{nJ\tau}\big\|\big(\M(\1_{Q_{J,w}\cap\Omega}\cdot|2^{ks}\phi_k\ast f|^\gamma)\big)_{k=J}^\infty\big\|_{L^\frac p\gamma(\R^n;\ell^\frac q\gamma)}^\frac1\gamma\lesssim\sup_{Q_{J,w}}2^{nJ\tau}\big\|\big(\1_{Q_{J,w}\cap\Omega}\cdot|2^{ks}\phi_k\ast f|^\gamma)\big)_{k=J}^\infty\big\|_{L^\frac p\gamma(\R^n;\ell^\frac q\gamma)}^\frac1\gamma
    \\
    &\quad=\sup_{Q_{J,w}}2^{nJ\tau}\big\|\big(\1_{\Omega}\cdot(2^{ks}\phi_k\ast f)\big)_{k=J}^\infty\big\|_{L^p(Q_{J,w};\ell^q)}=\big\|\big(2^{ks}\1_\Omega\cdot(\phi_k\ast f)\big)_{k=0}^\infty\big\|_{L^p_\tau\ell^q}.
\end{align*}
This completes the proof of \eqref{Eqn::PeeEst::SpeFs}.

The proof of  \eqref{Eqn::PeeEst::SpeBs} and \eqref{Eqn::PeeEst::SpeNs} are similar but simpler: by assumption $1<p/\gamma\le\infty$ we have
\begin{equation}\label{Eqn::PeeEst::BddM}
    \M:L^\frac p\gamma(\R^n)\to L^\frac p\gamma(\R^n).
\end{equation}
Therefore, we prove \eqref{Eqn::PeeEst::SpeBs} by the following:
\begin{align*}
    &\|(2^{js}\Pc^{\theta,N}_{\Omega,j}f)_{j=0}^\infty\|_{\ell^qL^p_\tau}
    \lesssim_{\theta,\phi,s,\tau,N,\gamma}\bigg\|\Big(\sum_{k=0}^\infty 2^{-(n\tau+1)\gamma|j-k|}\int_\Omega\frac{2^{kn}|2^{ks}\phi_k\ast f(y)|^\gamma dy}{(1+2^k|\cdot -y|)^{N\gamma}}\Big)_{j=0}^\infty\bigg\|_{\ell^\frac q\gamma L^{\frac p\gamma}_{\tau\gamma}}^\frac1\gamma&\text{by }\eqref{Eqn::STInq::omega}
    \\
    &\quad\lesssim_{p,q,s,\tau}\bigg\|\Big(\int_\Omega\frac{2^{kn}|2^{ks}\phi_k\ast f(y)|^\gamma dy}{(1+2^k|\cdot-y|)^{N\gamma}}\Big)_{k=0}^\infty\bigg\|_{\ell^\frac q\gamma L^\frac p\gamma_{\tau\gamma}}^\frac1\gamma&\text{by }\eqref{Eqn::ASumLem::Bs}
    \\
    &\quad\lesssim_{N,\gamma}\Big(\sum_{v\in\Z^n}\frac1{(1+|v|)^{N\gamma-n}}\Big)^{1/\gamma}\big\|\big(\M(\1_{\Omega}\cdot|2^{ks}\phi_k\ast f|^\gamma)\big)_{k=0}^\infty\big\|_{\ell^\frac q\gamma L^\frac p\gamma_{\tau\gamma}}^\frac1\gamma&\text{by }\eqref{Eqn::PeeToHL}
    \\
    &\quad\lesssim_{p,\gamma}\big\|\big(\1_{\Omega}\cdot|2^{ks}\phi_k\ast f|^\gamma)\big)_{k=0}^\infty\big\|_{\ell^{q/\gamma} L^{p/\gamma}_{\tau\gamma}}^{1/\gamma}=\big\|\big(2^{ks}\1_\Omega\cdot(\phi_k\ast f)\big)_{k=0}^\infty\big\|_{\ell^qL^p_\tau}&\text{by }\eqref{Eqn::PeeEst::BddM}.
\end{align*}

Finally we prove \eqref{Eqn::PeeEst::SpeNs}. Using \eqref{Eqn::STInq::Gamma=Infty} and \eqref{Eqn::ASumLem::Ns} (since $N>|s|+n\tau$) we have
\begin{equation}\label{Eqn::PeeEst::PfNs1}
    \|(2^{js}\Pc^{\theta,N}_{\Omega,j}f)_{j=0}^\infty\|_{\ell^qM^p_\tau}\lesssim_{\theta,\phi,s,N}\Big\|\Big(\sum_{k=0}^\infty2^{(N-|s|)|j-k|}2^{ks}\Pc^{\phi,N}_{\Omega,k}f\Big)_{j=0}^\infty\Big\|_{\ell^qM^p_\tau}\lesssim_{p,q,\tau,N}\|(2^{js}\Pc^{\phi,N}_{\Omega,j}f)_{j=0}^\infty\|_{\ell^qM^p_\tau}.
\end{equation}

Taking $\gamma\in(n/N,\min(p,q))$, we have $2^{js}(\Pc^{\phi,N}_{\Omega,j}f)\lesssim_{N,\gamma}\M(|2^{js}\1_\Omega\cdot(\phi_j\ast f)|^\gamma)^{1/\gamma}$ pointwise in $\R^n$.

When $p<\infty$ and $\tau<1/p$, by \cite[Lemma 2.5]{TangXuMorrey} we have
\begin{equation}\label{Eqn::PeeEst::PfNs2}
    \|2^{js}\Pc^{\phi,N}_{\Omega,j}f\|_{M^p_\tau}\lesssim_{N,\gamma}\big\|\M\big(|2^{js}\1_\Omega\cdot(\phi_j\ast f)|^\gamma\big)^{1/\gamma}\big\|_{M^p_\tau}\lesssim_{p,\gamma,\tau}\|2^{js}\1_\Omega\cdot(\phi_j\ast f)\|_{M^p_\tau},\quad j\ge0.
\end{equation}

We see that \eqref{Eqn::PeeEst::PfNs2} is valid for all $1<p/\gamma\le\infty,\tau\ge0$. 

When $\tau=1/p$, we have $M^p_\tau=L^\infty$ by \cite[Remark 11(ii)]{SickelMorrey1}, so \eqref{Eqn::PeeEst::PfNs2} follows from \eqref{Eqn::PeeEst::BddM}. When $\tau>1/p$ we have $M^p_\tau=\{0\}$, so \eqref{Eqn::PeeEst::PfNs2} holds trivially.

Thus by taking $\ell^q$-sum of \eqref{Eqn::PeeEst::PfNs2}, we get \eqref{Eqn::PeeEst::SpeNs}, completing the proof. 
\end{proof}

\begin{prop}\label{Prop::PeeEstRn}Let $\theta=(\theta_j)_{j=0}^\infty$ satisfies \ref{Item::LPCond::Momt} and \ref{Item::LPCond::Scal}, and let $\lambda=(\lambda_j)_{j=0}^\infty$ satisfies \ref{Item::CLPCond::Four} and \ref{Item::CLPCond::Scal}. For any $0<p,q\le\infty$, $s\in\R$, $\tau\ge 0$ and $N>\max(2n/\min(p,q),|s|+n\tau)$, there is a $C=C(\theta,\lambda,p,q,s,\tau,N)>0$ such that for every $\tilde f\in\Ss'(\R^n)$,
\begin{align}\label{Eqn::PeeEst::RnBs}
    \|(2^{js}\Pc^{\theta,N}_{\R^n,j}\tilde f)_{j=0}^\infty\|_{\ell^qL^p_\tau}&\le C\|(2^{js}\lambda_j\ast \tilde f)_{j=0}^\infty\|_{\ell^qL^p_\tau};
    \\\label{Eqn::PeeEst::RnFs}
    \|(2^{js}\Pc^{\theta,N}_{\R^n,j}\tilde f)_{j=0}^\infty\|_{L^p_\tau\ell^q}&\le C\|(2^{js}\lambda_j\ast \tilde f)_{j=0}^\infty\|_{L^p_\tau\ell^q},\qquad\text{provided }p<\infty;
    \\\label{Eqn::PeeEst::RnNs}
    \|(2^{js}\Pc^{\theta,N}_{\R^n,j}\tilde f)_{j=0}^\infty\|_{\ell^qM^p_\tau}&\le C\|(2^{js}\lambda_j\ast \tilde f)_{j=0}^\infty\|_{\ell^qM^p_\tau}.
\end{align}
\end{prop}
\begin{proof}
The proof is the same as that for Proposition \ref{Prop::PeeEst}, except that we replace every $\Omega$ by $\R^n$ in the arguments. We leave the details to readers.
\end{proof}

Based on Proposition \ref{Prop::PeeEst}, we can prove a boundedness result of Rychkov-type operators on $\As_{pq}^{s\tau}$-spaces.
\begin{prop}\label{Prop::Boundedness}
Let $\Omega\subset\R^n$ be a special Lipschitz domain and let $\gamma\in\R$. Let $\eta=(\eta_j)_{j=0}^\infty$ and $\theta=(\theta_j)_{j=0}^\infty$ satisfy conditions \ref{Item::LPCond::Momt}, \ref{Item::LPCond::Scal} and \ref{Item::LPCond::Supp} with respect to $\Omega$. We define an operator\footnote{The notation is slightly different from the one in \cite[Theorem 1.5]{ShiYaoExt}.} $T^{\eta,\theta,\gamma}_\Omega$ as
\begin{equation}\label{Eqn::ThmBdd::DefT}
    T^{\eta,\theta,\gamma}_\Omega f:=\sum_{j=0}^\infty 2^{j\gamma}\eta_j\ast(\1_\Omega\cdot(\theta_j\ast f)),\quad f\in\Ss'(\Omega).
\end{equation}
Then for $\As\in\{\Bs,\Fs,\Ns\}$, $0<p,q\le\infty$, $s\in\R$ and $\tau\ge0$ ($p<\infty$ for $\Fs$-cases), we have the boundedness
\begin{equation*}
    T^{\eta,\theta,\gamma}_\Omega:\As_{p,q}^{s,\tau}(\Omega)\to \As_{p,q}^{s-\gamma,\tau}(\R^n).
\end{equation*}
\end{prop}
\begin{proof}Recall $\Ss'(\Omega)=\{\tilde f|_\Omega:\tilde f\in\Ss'(\R^n)\}$ is defined via restrictions. We see that $T^{\eta,\theta,\gamma}_\Omega:\Ss'(\Omega)\to \Ss'(\R^n)$ is well-defined in the sense that, for every extension $\tilde f\in\Ss'(\R^n)$ of $f$, the summation $\sum_{j=0}^\infty 2^{j\gamma}\eta_j\ast(\1_\Omega\cdot(\theta_j\ast\tilde f))$ converges  $\Ss'(\R^n)$ and does not depend on the choice of $\tilde f$. See \cite[Propositions 3.10 and 3.14]{ShiYaoExt} for example.

Let $\lambda=(\lambda_j)_{j=0}^\infty$ be as in Definition \ref{Defn::Intro::DefMorrey} that defines the $\As_{pq}^{s\tau}$-norms. By Lemma \ref{Lem::Heideman}, for every $j,k\ge0$, $\int_{\R^n}|\lambda_j\ast\eta_k(y)|(1+2^k|y|)^Ndy\lesssim_{\lambda,\eta,N}2^{-N|j-k|}$. Thus by the similar argument to \eqref{Eqn::STInq::PfGamma=Infty}, for every $N>|s-\gamma|$,
\begin{align*}
        2^{j(s-\gamma)} 2^{k\gamma}|\lambda_j\ast \eta_k\ast (\1_\Omega \cdot (\theta_k\ast f))(x)|&\le2^{j(s-\gamma)} 2^{k\gamma}\int_\Omega|\lambda_j\ast \eta_k(y)|(1+2^k|y|)^Ndy\cdot\sup\limits_{t\in\Omega}\frac{|\theta_k\ast f(t)|}{(1+2^k|x-t|)^N}
        \\
        &\lesssim_{\lambda,\eta,N}2^{-(N-|s-\gamma|)|j-k|}2^{ks}(\Pc_{\Omega,k}^{\theta,N}f)(x).
    \end{align*}
    
    Therefore, by Lemma \ref{Lem::ASumLem}, for any $N>|s-\gamma|+n\tau$,
\begin{align}\label{Eqn::PfBdd::Sum1}
    \|(2^{j(s-\gamma)}\lambda_j\ast T^{\eta,\theta,\gamma}_\Omega  f)_{j=0}^\infty\|_{\ell^qL^p_\tau}&\lesssim_{\lambda,\eta,p,q,s,\gamma,\tau,N}\|(2^{ks}\Pc_{\Omega,k}^{\theta,N}f)_{k=0}^\infty\|_{\ell^qL^p_\tau};
    \\\label{Eqn::PfBdd::Sum2}
    \|(2^{j(s-\gamma)}\lambda_j\ast T^{\eta,\theta,\gamma}_\Omega  f)_{j=0}^\infty\|_{L^p_\tau\ell^q}&\lesssim_{\lambda,\eta,p,q,s,\gamma,\tau,N}\|(2^{ks}\Pc_{\Omega,k}^{\theta,N} f)_{k=0}^\infty\|_{L^p_\tau\ell^q},\qquad\text{provided }p<\infty;
    \\\label{Eqn::PfBdd::Sum3}
    \|(2^{j(s-\gamma)}\lambda_j\ast T^{\eta,\theta,\gamma}_\Omega  f)_{j=0}^\infty\|_{\ell^q M^p_\tau}&\lesssim_{\lambda,\eta,p,q,s,\gamma,\tau,N}\|(2^{ks}\Pc_{\Omega,k}^{\theta,N} f)_{k=0}^\infty\|_{\ell^q M^p_\tau}.
\end{align}

Let $\tilde f\in\As_{pq}^{s\tau}(\R^n)$ be an extension of $f$. Clearly $\Pc^{\theta,N}_{\Omega,k} f(x)=\Pc^{\theta,N}_{\Omega,k}\tilde f(x)\le\Pc^{\theta,N}_{\R^n,k}\tilde f(x)$ holds pointwise for $x\in\R^n$. Therefore, by choosing $N>2n/\min(p,q)$ and combining \eqref{Eqn::PfBdd::Sum1} and \eqref{Eqn::PeeEst::RnBs}, we have
\begin{equation*}
    \|T^{\eta,\theta,\gamma}_\Omega f\|_{\Bs_{pq}^{s\tau}(\R^n)}=\|(2^{j(s-\gamma)}\lambda_j\ast T^{\eta,\theta,\gamma}_\Omega  f)_{j=0}^\infty\|_{\ell^qL^p_\tau}\lesssim_{\eta,\theta,\lambda,p,q,s,\gamma,\tau}\|(2^{js}\lambda_j\ast \tilde f)_{j=0}^\infty\|_{\ell^qL^p_\tau}=\|\tilde f\|_{\Bs_{pq}^{s\tau}(\R^n)}.
\end{equation*}
Taking the infimum over all extensions $\tilde f$ of $f$ we get the boundedness $ T^{\eta,\theta,\gamma}_\Omega:\Bs_{p,q}^{s,\tau}(\Omega)\to\Bs_{p,q}^{s-\gamma,\gamma}(\R^n)$. Similarly using \eqref{Eqn::PfBdd::Sum2}, \eqref{Eqn::PeeEst::RnFs} and \eqref{Eqn::PfBdd::Sum3}, \eqref{Eqn::PeeEst::RnNs} we get $T^{\eta,\theta,\gamma}_\Omega:\As_{p,q}^{s,\tau}(\Omega)\to\As_{p,q}^{s-\gamma,\gamma}(\R^n)$ for $\As\in\{\Fs,\Ns\}$.
\end{proof}
\begin{rem}
Under the definition \eqref{Eqn::Intro::NormDomain}, the operator norms of $T^{\eta,\theta,\gamma}_\Omega$ do not depend\footnote{It can depend on the upper bound of $\|\nabla\rho\|_{L^\infty}$, which is bounded by $\inf\{-\frac{x_n}{|x'|}:(x',x_n)\in\supp\phi_j\}$ where $\phi\in\{\eta,\theta\}$ and $j\ge0$.} on $\Omega$. This is due to the same reason as mentioned in \cite[Remark 3.11]{ShiYaoExt}:

One can see that the constants in Proposition \ref{Prop::PeeEst} depend on everything except on $\Omega$. The same hold for the implied constants in \eqref{Eqn::PfBdd::Sum1}, \eqref{Eqn::PfBdd::Sum2} and \eqref{Eqn::PfBdd::Sum3}. After the pointwise inequality $\Pc^{\theta,N}_{\Omega,k} f\le\Pc^{\theta,N}_{\R^n,k}\tilde f$, it remains to estimate $(2^{js}\Pc^{\theta,N}_{\R^n,j}\tilde f)_{j=0}^\infty$ (which is Proposition \ref{Prop::PeeEstRn}), where $\Omega$ is not involved.
\end{rem}

\begin{cor}[\cite{YSYIntBesovMorrey,ZhuoTriebelType,ZhuoBesovType}]\label{Cor::BddExt}
Let $\Omega\subset\R^n$ be a special Lipschitz domain. Let $\phi=(\phi_j)_{j=0}^\infty$ and $\psi=(\psi_j)_{j=0}^\infty$ be as in the assumption and conclusion of Lemma \ref{Lem::ExistPsi} with respect to $\Omega$. Then the Rychkov's extension operator
\begin{equation}\label{Eqn::ExtOp}
E_\Omega f=E^{\psi,\phi}_\Omega f: =\sum_{j=0}^\infty\psi_j\ast(\1_{\Omega}\cdot(\phi_j\ast f)),\quad f\in\Ss'(\Omega),
\end{equation}
is well-defined and has boundedness $E_\Omega:\As_{pq}^{s\tau}(\Omega)\to\As_{pq}^{s\tau}(\R^n)$ for $\As\in\{\Bs,\Fs,\Ns\}$ and all $0<p,q\le\infty$, $s\in\R$, $\tau\ge0$ ($p<\infty$ for $\Fs$-cases).
\end{cor}

\begin{proof}$E_\Omega$ is an extension operator because by assumption $E_\Omega f|_\Omega=\sum_{j=0}^\infty\psi_j\ast\phi_j \ast f=f$.
The boundedness is immediate since $E_\Omega=T^{\psi,\phi,0}_\Omega$ from \eqref{Eqn::ThmBdd::DefT}.
\end{proof}

\begin{rem}
Corollary \ref{Cor::BddExt} is not new. See \cite[Proposition 4.13]{YSYIntBesovMorrey} for $\As=\Ns$, \cite[Section 4]{ZhuoTriebelType} for $\As=\Fs$ and \cite[Section 4]{ZhuoBesovType} for $\As=\Bs$. For the proof we also refer \cite[Theorem 3.6]{GHSMorreyExtension} to readers.
\end{rem}

The key to prove Theorem \ref{Thm::LPNorm} is to use the following analog of \cite[Theorem 2.3]{ExtensionLipschitz}. 
\begin{prop}[Characterizations via Peetre's maximal functions]\label{Prop::NormMax}
Let $\Omega\subset\R^n$ be a special Lipschitz domain and let $\phi=(\phi_j)_{j=0}^\infty$ be a Littlewood-Paley family associated with $\Omega$. Then for $0<p,q\le\infty$, $s\in\R$ and $\tau\ge0$ ($p<\infty$ for $\Fs$-cases), we have the following intrinsic characterizations: for every $N>\max(\frac{2n}{\min(p,q)},|s|+n\tau)$,
\begin{align}\label{Eqn::NormMax::Bs}
    \|f\|_{\Bs_{pq}^{s\tau}(\Omega)}&\approx_{\phi,p,q,s,\tau,N}\big\|\big(2^{js}\1_\Omega\cdot(\Pc^{\phi,N}_{\Omega,j}f)\big)_{j=0}^\infty\big\|_{\ell^qL^p_\tau};
    \\\label{Eqn::NormMax::Fs}
    \|f\|_{\Fs_{pq}^{s\tau}(\Omega)}&\approx_{\phi,p,q,s,\tau,N}\big\|\big(2^{js}\1_\Omega\cdot(\Pc^{\phi,N}_{\Omega,j}f)\big)_{j=0}^\infty\big\|_{L^p_\tau\ell^q},\qquad\text{provided }p<\infty;
    \\
    \label{Eqn::NormMax::Ns}
    \|f\|_{\Ns_{pq}^{s\tau}(\Omega)}&\approx_{\phi,p,q,s,\tau,N}\big\|\big(2^{js}\1_\Omega\cdot(\Pc^{\phi,N}_{\Omega,j}f)\big)_{j=0}^\infty\big\|_{\ell^q M^p_\tau}.
\end{align}
\end{prop}
\begin{rem}
\eqref{Eqn::NormMax::Bs} and \eqref{Eqn::NormMax::Fs} are not new as well. The case $\As=\Fs$ is done in \cite[Theorem 1.7]{SunZhuoExtension}, where a more general setting is considered. See also \cite[Proof of Theorem 3.6, Step 2]{GHSMorreyExtension} for a proof of $\As\in\{\Bs,\Fs\}$.

As already mentioned in Remark \ref{Rmk::RmkN}, it is possible that the assumption of $N$ can be weakened.
\end{rem}
\begin{proof}[Proof of Proposition \ref{Prop::NormMax}]
Let $\lambda=(\lambda_j)_{j=0}^\infty$ be as in Definition \ref{Defn::Intro::DefMorrey} that defines the $\As_{pq}^{s\tau}$-norms. We only prove \eqref{Eqn::NormMax::Fs} since the proof of \eqref{Eqn::NormMax::Bs} and \eqref{Eqn::NormMax::Ns} are the same by replacing $L^p_\tau\ell^q$ with $\ell^qL^p_\tau$ and $\ell^q M^p_\tau$, and including the discussion of $p=\infty$.

\medskip
\noindent($\gtrsim$) For $f\in\Fs_{pq}^{s\tau}(\Omega)$, let $\tilde f\in\Fs_{pq}^{s\tau}(\R^n)$ be an extension of $f$. We see that pointwisely 
$$(\1_\Omega\cdot\Pc^{\phi,N}_{\Omega,j}f)(x)\le\Pc^{\phi,N}_{\Omega,j}f(x)= \Pc^{\phi,N}_{\Omega,j}\tilde f(x)\le \Pc^{\phi,N}_{\R^n,j}\tilde f(x),\quad j\ge0,\quad x\in\R^n.$$
Thus by Proposition \ref{Prop::PeeEst},
\begin{equation*}
    \big\|\big(2^{js}\1_\Omega\cdot(\Pc^{\phi,N}_jf)\big)_{j=0}^\infty\big\|_{L^p_\tau\ell^q}\le \big\|\big(2^{js}\Pc^{\phi,N}_{\Omega,j}\tilde f\big)_{j=0}^\infty\big\|_{L^p_\tau\ell^q}\lesssim_{\lambda,\phi,p,q,s,\gamma,\tau,N}\big\|\big(2^{js}\lambda_j\ast\tilde f\big)_{j=0}^\infty\big\|_{L^p_\tau\ell^q}=\|\tilde f\|_{\Fs_{pq}^{s\tau}(\R^n)}.
\end{equation*}
Taking infimum over all extensions $\tilde f$ of $f$, we get  $\|f\|_{\Fs_{pq}^{s\tau}(\Omega)}\gtrsim\big\|\big(2^{js}\1_\Omega\cdot(\Pc^{\phi,N}_{\Omega,j}f)\big)_{j=0}^\infty\big\|_{L^p_\tau\ell^q}$.

\medskip
\noindent($\lesssim$) By Corollary \ref{Cor::BddExt} we have $\|f\|_{\Fs_{pq}^{s\tau}(\Omega)}\approx\|E_\Omega f\|_{\Fs_{pq}^{s\tau}(\R^n)}=\|(2^{js}\lambda_j\ast E_\Omega f)_{j=0}^\infty\|_{L^p_\tau\ell^q}$. Therefore using \eqref{Eqn::PfBdd::Sum1} with the fact that $E_\Omega=T^{\psi,\phi,0}_\Omega$, 
\begin{equation}\label{Eqn::NormMax::Tmp1}
    \|(2^{js}\lambda_j\ast E_\Omega f)_{j=0}^\infty\|_{L^p_\tau\ell^q}=\|(2^{js}\lambda_j\ast T^{\psi,\phi,0}_\Omega f)_{j=0}^\infty\|_{L^p_\tau\ell^q}\lesssim_{\psi,\phi,\lambda,p,q,s,\tau}\|(2^{js}\Pc^{\phi,N}_{\Omega,j} f)_{j=0}^\infty\|_{L^p_\tau\ell^q}.
\end{equation}

Write $\Omega=\{(x',x_n):x_n>\rho(x')\}$. We define a ``fold map'' $L=L_\Omega:\R^n\twoheadrightarrow\overline\Omega$ as
\begin{equation*}
    L(x):=x\quad\text{if }x\in\Omega;\qquad L(x):=(x',2\rho(x')-x_n),\quad\text{if }x\notin\Omega.
\end{equation*}

 Recall $\Omega=\{x_n>\rho(x')\}$. By direct computation, we have
\begin{equation}\label{Eqn::NormMax::Tmp1.5}
    \textstyle|L(x)-y|\le \big(\|\nabla\rho\|_{L^\infty}+\sqrt{1+\|\nabla\rho\|_{L^\infty}^2}\big)^2|x-y|\lesssim_\Omega|x-y|\qquad x\in\R^n,\quad y\in\Omega.
\end{equation}

Therefore \begin{equation*}
    \Pc_{\Omega,j}^{\phi,N}f(x)=\sup_{y\in\Omega}\frac{|\phi_j\ast f(y)|}{(1+2^j|x-y|)^N}\lesssim_{\Omega,N}\sup_{y\in\Omega}\frac{|\phi_j\ast f(y)|}{(1+2^j|L(x)-y|)^N}=\big(\Pc_{\Omega,j}^{\phi,N}f\big)\big(L(x)\big),\quad x\in\R^n.
\end{equation*}

Clearly for $0<p\le\infty$ we have the following estimate for cube $Q\in\Qc$ and function $g\in L^p_\loc(\Omega)$:
\begin{equation*}
    \|g\circ L\|_{L^p(Q)}\lesssim_p\|g\|_{L^p(\Omega\cap L^{-1}(Q))}\lesssim_p\sum_{P\in \mathcal I_Q}\|\1_\Omega\cdot g\|_{L^p(P)},\quad\text{where }\mathcal I_Q:=\{P\in\Qc:|P|=|Q|,\ P\cap L^{-1}(Q)\neq\varnothing\}.
\end{equation*}

By \eqref{Eqn::NormMax::Tmp1.5} we have control of the cardinality $\# \mathcal I_Q\lesssim_n(1+\|\nabla\rho\|_{L^\infty})^{2n}\lesssim_\Omega1$, which is uniform in $Q\in\Qc$. Therefore,
\begin{equation}\label{Eqn::NormMax::Tmp2}
    \|(2^{js}\Pc^{\phi,N}_{\Omega,j} f)_{j=0}^\infty\|_{L^p_\tau\ell^q}\lesssim_N\big\|\big(2^{js}(\Pc^{\phi,N}_{\Omega,j} f)\circ L\big)_{j=0}^\infty\big\|_{L^p_\tau\ell^q}\lesssim_{p,q,\Omega}\big\|\big(2^{js}\1_\Omega\cdot(\Pc^{\phi,N}_{\Omega,j} f)\big)_{j=0}^\infty\big\|_{L^p_\tau\ell^q}.
\end{equation}
Combining \eqref{Eqn::NormMax::Tmp1} and \eqref{Eqn::NormMax::Tmp2} we get $\|f\|_{\Fs_{pq}^{s\tau}(\Omega)}\lesssim\big\|\big(2^{js}\1_\Omega\cdot(\Pc^{\phi,N}_{\Omega,j}f)\big)_{j=0}^\infty\big\|_{L^p_\tau\ell^q}$, finishing the proof.
\end{proof}

We can now prove Theorem \ref{Thm::LPNorm}:
\begin{proof}[Proof of Theorem \ref{Thm::LPNorm}]The $\Fs_{\infty q}^s$-cases follow immediately from the $\Fs_{pq}^{s\tau}$-cases using \eqref{Eqn::Intro::CharFInftyQ}.

Fix a $N>\max(2n/\min(p,q),|s|+n\tau)$. We only prove the $\Fs_{pq}^{s\tau}$-cases. The proofs of the $\Bs_{pq}^{s\tau}$-cases and the $\Ns_{pq}^{s\tau}$-cases are the same, except that we replace every $L^p_\tau\ell^q$ with $\ell^qL^p_\tau$ and $\ell^q M^p_\tau$.

By Proposition \ref{Prop::NormMax} we have $\|f\|_{\Fs_{pq}^{s\tau}(\Omega)}\approx\big\|\big(2^{js}\1_\Omega\cdot(\Pc^{\phi,N}_{\Omega,j}f)\big)_{j=0}^\infty\big\|_{L^p_\tau\ell^q}$. Therefore, it suffices to show that $\big\|\big(2^{js}\1_\Omega\cdot(\Pc^{\phi,N}_{\Omega,j}f)\big)_{j=0}^\infty\big\|_{L^p_\tau\ell^q}\approx \big\|\big(2^{js}\1_\Omega\cdot(\phi_j\ast f)\big)_{j=0}^\infty\big\|_{L^p_\tau\ell^q}$.

Clearly $\big\|\big(2^{js}\1_\Omega\cdot(\Pc^{\phi,N}_{\Omega,j}f)\big)_{j=0}^\infty\big\|_{L^p_\tau\ell^q}\ge \big\|\big(2^{js}\1_\Omega\cdot(\phi_j\ast f)\big)_{j=0}^\infty\big\|_{L^p_\tau\ell^q}$ since $\phi_j\ast f(x)\le \Pc^{\phi,N}_{\Omega,j}f(x)$ holds for all $f\in\Ss'(\Omega)$, $x\in\Omega$ and $j\ge0$.
The converse $\big\|\big(2^{js}\1_\Omega\cdot(\Pc^{\phi,N}_{\Omega,j}f)\big)_{j=0}^\infty\big\|_{L^p_\tau\ell^q}\lesssim_{\phi,p,q,s,\tau,N} \big\|\big(2^{js}\1_\Omega\cdot(\phi_j\ast f)\big)_{j=0}^\infty\big\|_{L^p_\tau\ell^q}$ follows from \eqref{Eqn::PeeEst::SpeBs}. Thus, we prove the $\Fs_{pq}^{s\tau}$-cases.
\end{proof}

We have the immediate analogy of \cite[Theorem 1.1]{YangYuan10} on Lipschitz domains: 
\begin{cor}\label{Cor::NormMaxP}
Keeping the assumptions in Proposition \ref{Prop::NormMax}, we have the following intrinsic characterizations: for every $N>\max(2n/\min(p,q),|s|+n\tau)$,
\begin{align*}
    \|f\|_{\Bs_{pq}^{s\tau}(\Omega)}&\approx_{\phi,p,q,s,\tau,N}\sup_{Q_{J,v}\in\Qc}2^{nJ\tau}\Big(\sum_{j=\max(0,J)}^\infty2^{jsq}\|\Pc^{\phi,N}_{(Q_{J,v}\cap\Omega),j}f\|_{L^p(Q_{J,v}\cap\Omega)}^q\Big)^\frac1q;
    \\
    \|f\|_{\Fs_{pq}^{s\tau}(\Omega)}&\approx_{\phi,p,q,s,\tau,N}\sup_{Q_{J,v}\in\Qc}2^{nJ\tau}\Big(\int_{ Q_{J,v}\cap\Omega}\Big(\sum_{j=\max(0,J)}^\infty2^{jsq}|\Pc^{\phi,N}_{(Q_{J,v}\cap\Omega),j}f(x)|^q\Big)^\frac pqdx\Big)^\frac1p,\quad \text{provided }p<\infty;
    \\
    \|f\|_{\Ns_{pq}^{s\tau}(\Omega)}&\approx_{\phi,p,q,s,\tau,N}\Big(\sum_{j=0}^\infty\sup_{Q_{J,v}\in\Qc}2^{nJ\tau q+jsq}\|\Pc^{\phi,N}_{(Q_{J,v}\cap\Omega),j}f\|_{L^p(Q_{J,v}\cap\Omega)}^q\Big)^\frac 1q.
\end{align*}
\end{cor}
%Recall the notation $\Pc^{\phi,N}_{Q_{Jv}\cap\Omega,j}$ from \eqref{Eqn::PeeMaxDef}.
\begin{proof}
Since $|\phi_j\ast f(x)|\le\Pc^{\phi,N}_{(Q_{J,v}\cap\Omega),j}f(x)\le \Pc^{\phi,N}_{\Omega,j}f(x)$ pointwisely for every $Q_{J,v}\in\mathcal Q$ and $x\in Q_{J,v}\cap\Omega$, the results follow immediately by combining Theorem \ref{Thm::LPNorm} and Proposition \ref{Prop::NormMax}. 
\end{proof}

\begin{rem}\label{Rmk::RmkBddDom}
By the standard partition of unity argument, we can give the analogy of Theorem \ref{Thm::LPNorm} on a bounded Lipschitz domain. An example is the following:
\begin{align}\label{Eqn::BddDomLPNorm::B}
    \|f\|_{\Bs_{pq}^{s\tau}(\Omega)}&\approx\sum_{\nu=1}^N\|(2^{js}\1_{\Omega\cap U_\nu}\cdot(\phi_j^\nu\ast(\chi_\nu f)))_{j=0}^\infty\|_{\ell^qL^p_\tau};
    \\
    \label{Eqn::BddDomLPNorm::F}
    \|f\|_{\Fs_{pq}^{s\tau}(\Omega)}&\approx\sum_{\nu=1}^N\|(2^{js}\1_{\Omega\cap U_\nu}\cdot(\phi_j^\nu\ast(\chi_\nu f)))_{j=0}^\infty\|_{L^p_\tau\ell^q};
    \\
    \label{Eqn::BddDomLPNorm::N}
    \|f\|_{\Ns_{pq}^{s\tau}(\Omega)}&\approx\sum_{\nu=1}^N\|(2^{js}\1_{\Omega\cap U_\nu}\cdot(\phi_j^\nu\ast(\chi_\nu f)))_{j=0}^\infty\|_{\ell^qM^p_\tau};
\end{align}

Here $\{U_\nu,(\phi_j^\nu)_{j=0}^\infty,\chi_\nu\}_{\nu=1}^N$ satisfy the following:
\begin{itemize}
    \item $\{U_\nu\}_{\nu=1}^N$ is an open cover of $\overline{\Omega}$, and there are cones $K_\nu\subset\R^n$ such that $U_\nu\cap(\Omega-K_\nu)\subseteq U_\nu\cap\Omega$ for each $\nu=1,\dots,N$.
    \item For $\nu=1,\dots,N$, $(\phi_j^\nu)_{j=0}^\infty$ satisfies \ref{Item::LPCond::Momt} - \ref{Item::LPCond::AppId} in Definition \ref{Defn::Intro::Basic}, with support condition $\supp\phi_j^\nu\subset K_\nu$ for $j\ge0$.
    \item $\chi_\nu\in C_c^\infty(U_\nu)$ for $\nu=1,\dots,N$, and satisfy\footnote{In fact we can relax the condition to $\sum_{\nu=1}^N\chi_\nu|_{\overline{\Omega}}>c$ for some $c>0$.} $\sum_{\nu=1}^N\chi_\nu|_{\overline{\Omega}}\equiv1$.
\end{itemize}

To prove \eqref{Eqn::BddDomLPNorm::B}, \eqref{Eqn::BddDomLPNorm::F} and \eqref{Eqn::BddDomLPNorm::N} the only thing we need are the following standard results ($p<\infty$ for $\Fs$-cases):
\begin{enumerate}[label=($\Psi$.\alph*)]
    \item \label{Item::Property::MultBdd} Let $\chi\in C_c^\infty(\R^n)$. Then $[\tilde f\mapsto \chi \tilde f]:\As_{pq}^{s\tau}(\R^n)\to\As_{pq}^{s\tau}(\R^n)$ is bounded.
    \item \label{Item::Property::CompBdd} Let $\Phi$ be an invertible affine linear transform. Then $[\tilde f\mapsto \tilde f\circ\Phi]:\As_{pq}^{s\tau}(\R^n)\to\As_{pq}^{s\tau}(\R^n)$ is bounded.
    \item\label{Item::Property::DerNorm} For every $m\ge1$, we have equivalent norms $\|f\|_{\As_{p,q}^{s,\tau}(\R^n)}\approx_{p,q,s,\tau,m}\sum_{|\alpha|\le m}\|\partial^\alpha f\|_{\As_{p,q}^{s-m,\tau}(\R^n)}$.
\end{enumerate}
One can see \cite[Sections 6.1.1 and 6.2]{YSYMorrey}, \cite[Theorem 1.6]{WuYangYuanDer} and \cite[Theorem 3.3]{SawanoTanakaMorrey}  for their proof. See also \cite[Sections 3.4, 4.2 and 4.3]{HaroskeTriebelMorrey}. We remark that because of \eqref{Eqn::Intro::EqvSpaces} it is enough to consider the case $0\le \tau\le\frac1p$. We leave the details to the readers.

One can also write down the analogy of Proposition \ref{Prop::NormMax} and Corollary \ref{Cor::NormMaxP} similar to \eqref{Eqn::BddDomLPNorm::B}, \eqref{Eqn::BddDomLPNorm::F} and \eqref{Eqn::BddDomLPNorm::N}, we leave the details to the readers as well.
\end{rem}

Finally, we prove Theorem \ref{Thm::EqvNorm} using the following fact:
\begin{prop}[{\cite[Theorem 1.5 (ii)]{ShiYaoExt}}]\label{Prop::ShiYaoFact}
    Let $(\phi_j)_{j=1}^\infty$ be a family\footnote{Here the index of the Schwartz family start from $j=1$. In Definition \ref{Defn::Intro::SeqSpaces} we start with $j=0$.} of Schwartz functions satisfying \ref{Item::LPCond::Momt}, \ref{Item::LPCond::Scal} and \ref{Item::LPCond::Supp}. Recall that for every $j\ge1$, $\phi_j(x)=2^{(j-1)n}\phi_1(2^{j-1}x)$, $\int x^\alpha \phi_j(x)dx=0$ for all $\alpha$, and $\supp\phi_j\subset\{x_n<-A|x'|\}$ for some $A>0$.
    
    Then for any $m\ge1$, there are  families of Schwartz functions $\tilde\phi^\beta=(\tilde\phi^\beta_j)_{j=1}^\infty$ for $|\beta|=m$ that also satisfy \ref{Item::LPCond::Momt}, \ref{Item::LPCond::Scal} and \ref{Item::LPCond::Supp}, such that $$\phi_j=2^{-jm}\sum_{|\beta|=m}\partial^\beta\tilde\phi_j^\beta,\quad\text{for every }j\ge1.$$
\end{prop}

\begin{proof}[Proof of Theorem \ref{Thm::EqvNorm}]

Once the case of special Lipschitz domains is done, the proof of the case of bounded Lipschitz domains follows from the standard partition of unity argument (one can read \cite[Section 6]{ShiYaoExt} for details) along with the facts \ref{Item::Property::MultBdd}, \ref{Item::Property::CompBdd} and \ref{Item::Property::DerNorm} mentioned in Remark \ref{Rmk::RmkBddDom}.

\medskip
Let $\Omega\subset\R^n$ be a special Lipschitz domain. Let $f\in\As_{pq}^{s\tau}(\Omega)$ and let $\tilde f\in\As_{pq}^{s\tau}(\R^n)$ be an extension of $f$.  By \ref{Item::Property::DerNorm}  we have $\|\partial^\alpha \tilde f\|_{\As_{p,q}^{s-|\alpha|,\tau}(\R^n)}\lesssim_{p,q,s,\tau,\alpha}\|\tilde f\|_{\As_{pq}^{s\tau}(\R^n)}$.
Since $\partial^\alpha \tilde f$ is also an extension of $\partial^\alpha f$, by \eqref{Eqn::Intro::NormDomain} in Definition \ref{Defn::Intro::SeqSpaces}, taking the infimum over all extensions $\tilde f$ of $f$ we get $\sum_{|\alpha|\le m}\|\partial^\alpha f\|_{\As_{p,q}^{s-m,\tau}(\Omega)}\lesssim\|f\|_{\As_{pq}^{s\tau}(\Omega)}$.

To prove the converse inequality $\|f\|_{\As_{pq}^{s\tau}(\Omega)}\lesssim\sum_{|\alpha|\le m}\|\partial^\alpha f\|_{\As_{p,q}^{s-m,\tau}(\Omega)}$, let $(\phi_j,\psi_j)_{j=0}^\infty$ be as in \eqref{Eqn::ExtOp}. 

We let $(\tilde\phi^\beta_j)_{j=1}^\infty\subset\Ss(\R^n)$ ($|\beta|>0$) be given in Proposition \ref{Prop::ShiYaoFact}. Thus $\phi_j=2^{-jq}\sum_{\beta:|\beta|=q}\partial^\beta\tilde\phi^\beta_j$ for all $j,q\ge1$.

For $\alpha\neq0$, we define $\psi^\alpha=(\psi^\alpha_j)_{j=1}^\infty$ by $\psi^\alpha_j(x):=2^{-j|\alpha|}\partial^\alpha\psi_j(x)$ (for $j\ge1$). Thus the sequences $\psi^\alpha$ (for $\alpha\neq0$) all satisfy \ref{Item::LPCond::Momt}, \ref{Item::LPCond::Scal} and \ref{Item::LPCond::Supp}.

We define a family of linear operators,
\begin{equation}\label{Eqn::EAlphaBeta}
    E^{\alpha,0} f=E^{\alpha,0}_\Omega f:=\partial^\alpha\psi_0\ast(\1_\Omega\cdot(\phi_0\ast f)),\quad E^{\alpha,\beta}f=E^{\alpha,\beta}_\Omega f:=\sum_{j=1}^\infty \psi_j^\alpha\ast(\1_\Omega\cdot(\tilde\phi_j^\beta\ast f)),\text{ for }|\alpha|=|\beta|>0.
\end{equation}

For every $f\in\Ss'(\Omega)$ and for every multi-index $\alpha\neq0$, we see that
\begin{equation}\label{Eqn::EqvNormPfKey}
    \begin{aligned}
    \partial^\alpha E f&=\sum_{j=0}^\infty\partial^\alpha\psi_j\ast(\1_\Omega\cdot(\phi_j\ast f))=\partial^\alpha\psi_0\ast(\1_\Omega\cdot(\phi_0\ast f))+\sum_{j=1}^\infty
    \sum_{\beta:|\beta|=|\alpha|}2^{j|\alpha|}\psi_j^\alpha\ast(\1_\Omega\cdot2^{-j|\alpha|}(\partial^\beta \tilde\phi_j^\beta\ast f))
    \\
    &=\partial^\alpha\psi_0\ast(\1_\Omega\cdot(\phi_0\ast f))+\sum_{\beta:|\beta|=|\alpha|}\sum_{j=1}^\infty \psi_j^\alpha\ast(\1_\Omega\cdot(\tilde\phi_j^\beta\ast\partial^\beta f))=E^{\alpha,0}f+\sum_{\beta:|\beta|=|\alpha|}E^{\alpha,\beta}[\partial^\beta f].
\end{aligned}
\end{equation}

By Proposition \ref{Prop::Boundedness}, $E^{\alpha,0},E^{\alpha,\beta}:\As_{p,q}^{s-m,\tau}(\Omega)\to\As_{p,q}^{s-m,\tau}(\R^n)$ are all bounded. Therefore
\begin{equation*}
    \begin{aligned}
    \|f\|_{\As_{pq}^{s\tau}(\Omega)}\approx&\|E f\|_{\As_{pq}^{s\tau}(\R^n)}\overset{\text{\ref{Item::Property::DerNorm}}}{\approx}\sum_{|\alpha|\le m} \|\partial^\alpha Ef\|_{\As_{pq}^{s-m,\tau}(\R^n)}
    \\
    \overset{\text{\eqref{Eqn::EqvNormPfKey}}}\lesssim&\| Ef\|_{\As_{pq}^{s-m,\tau}(\R^n)}+\sum_{0<|\alpha|\le m}\Big(\| E^{\alpha,0}f\|_{\As_{pq}^{s-m,\tau}(\R^n)}+\sum_{\beta:|\beta|=|\alpha|}\| E^{\alpha,\beta}[\partial^\beta f]\|_{\As_{pq}^{s-m,\tau}(\R^n)}\Big)
    \\
    \lesssim&\|f\|_{\As_{pq}^{s-m,\tau}(\Omega)}+\sum_{0<|\alpha|\le m}\Big(\| f\|_{\As_{pq}^{s-m,\tau}(\Omega)}+\sum_{\beta:|\beta|=|\alpha|}\|\partial^\beta f\|_{\As_{pq}^{s-m,\tau}(\Omega)}\Big)\lesssim\sum_{|\beta|\le m}\|\partial^\beta f\|_{\As_{pq}^{s-m,\tau}(\Omega)}.
\end{aligned}
\end{equation*}

This completes the proof of \eqref{Eqn::Intro::EqvNorm} for the case of special Lipschitz domains.

The $\Fs_{\infty q}^s$-cases follow immediately from \eqref{Eqn::Intro::CharFInftyQ} since we have $\Fs_{\infty q}^s(\R^n)=\Fs_{qq}^{s,\frac1q}(\R^n)$.
\end{proof}

\section{Further Open Questions}\label{Sec::Rmk}
By the same method, using Lemma \ref{Lem::Heideman} - Proposition \ref{Prop::PeeEst}, it is possible for us to get the analogs of Theorems \ref{Thm::LPNorm} and \ref{Thm::EqvNorm} on the so-called \textit{local spaces}. 

The local version of $\As_{pq}^{s\tau}(\R^n)$ for $\As\in\{\Bs,\Fs,\Ns\}$, denoted by $\As_{p,q,\text{unif}}^{s,\tau}(\R^n)$,  is defined by replacing the supremum among the set of dyadic cubes $\Qc$ with $\{Q_{J,v}\in\Qc:J\ge0\}$. See \cite[Section 3.4]{SickelMorrey1} for example. For an open subset $\Omega\subseteq\R^n$ we use $\As_{p,q,\text{unif}}^{s,\tau}(\Omega):=\{\tilde f|_\Omega:\tilde f\in\As_{p,q,\text{unif}}^{s,\tau}(\R^n)\}$ similarly. For more details we refer \cite{TriebelLocal} to readers.

One can also consider the analog of Theorems \ref{Thm::LPNorm} and \ref{Thm::EqvNorm} on $\As_{p(\cdot),q(\cdot)}^{s(\cdot),\phi}$, the spaces with variable exponents. For example \cite{SunZhuoExtension}, which may require certain assumptions on the exponents.

\medskip
In Definition \ref{Defn::Intro::DefMorrey}, it is known that the norms are equivalent if $(\lambda_j)_{j=0}^\infty$ only satisfies the scaling condition \ref{Item::LPCond::Scal} and the \textit{Tauberian condition}:
    \begin{equation}\label{Eqn::LPCond::Taub}
        \text{There exist }\eps_0,c>0\quad\text{such that}\quad|\hat\lambda_0(\xi)|>c\text{ for }|\xi|<\eps_0,\quad\text{and }|\hat\lambda_1(\xi)|>c\text{ for }\eps_0/2<|\xi|<2\eps_0.
    \end{equation}
    See \cite[Theorems 2.5 and 2.6]{WuYangYuanDer} and \cite[Theorem 1]{XuBesovMorrey} for example.
    
    It is not known to the author whether we can replace the assumption \ref{Item::LPCond::AppId} for $(\phi_j)_{j=0}^\infty$ in Theorem \ref{Thm::LPNorm} with the Tauberian condition \eqref{Eqn::LPCond::Taub}.
    
\medskip
For Theorem \ref{Thm::EqvNorm}, we do not know whether \eqref{Eqn::Intro::EqvNorm} has the following improvement:
\begin{ques}
Keeping the assumptions of Theorem \ref{Thm::EqvNorm}, can we find a $C=C(\Omega,p,q,s,\tau,m)>0$ such that the following holds? 
\begin{equation*}
    \|f\|_{\As_{pq}^{s\tau}(\Omega)}\le C\Big(\|f\|_{\As_{p,q}^{s-m,\tau}(\Omega)}+\sum_{k=0}^n\Big\|\frac{\partial^m f}{\partial x_k^m}\Big\|_{\As_{p,q}^{s-m,\tau}(\Omega)}\Big),\quad\forall\,f\in\As_{pq}^{s\tau}(\Omega).
\end{equation*}
\end{ques}

Cf. \cite[Theorem 1.6]{WuYangYuanDer}. The question is open even for the classical Besov and Triebel-Lizorkin spaces $\As_{pq}^s(\Omega)$ when $\Omega$ is a (special or bounded) Lipschitz domain.

\begin{ack}
I  would like to thank Dorothee Haroske,  Wen Yuan and Ciqiang Zhuo for their informative discussions and advice.  I would also like to thank the referees for the comments and suggestions.
\end{ack}

\bibliographystyle{amsalpha}
\bibliography{reference} 
\end{document}